\documentclass[10pt]{amsart}
\usepackage{amsthm,amsmath,amssymb,amsfonts}
\usepackage{graphicx}
\newtheorem{lemma}{Lemma}[section]

\newtheorem{theorem}[lemma]{Theorem}
\newtheorem{corollary}[lemma]{Corollary}
\newtheorem{definition}[lemma]{Definition}

\theoremstyle{definition}
\newtheorem{example}[lemma]{Example}

\newcommand{\cH}{\mathcal{ H}}

\newcommand{\cL}{\mathcal{L}}

\newcommand{\Dom}{{\rm Dom}}

\newcommand{\Spec}{{\rm Spec}}
\newcommand{\Span}{{\rm span}}

\renewcommand{\Re}{{\rm Re}\;}
\renewcommand{\Im}{{\rm Im}\;}

\newcommand{\conv}{{\rm conv}}
\newcommand{\dist}{{\rm dist}}

\newcommand{\ess}{\mathrm{ess}}

\newcommand{\range}{\mathrm{Range}}
\newcommand{\dis}{\mathrm{dis}}
\newcommand{\intr}{\mathrm{int}}

\begin{document}
\title{The Second Order Spectrum and Optimal Convergence}
\author{Michael Strauss}
\date{April 2011}
\maketitle
\begin{abstract}
The method of second order relative spectra has been shown to reliably approximate the discrete spectrum
for a self-adjoint operator. We extend the method to normal operators and find optimal convergence rates for
eigenvalues and eigenspaces. The convergence to eigenspaces is new,
while the convergence rate for eigenvalues improves on the previous estimate by an order of magnitude.
\end{abstract}
\section{Introduction}
Throughout this manuscript $A$ will be a normal linear operator acting on an infinite dimensional Hilbert
space $\mathcal{H}$. The domain, spectrum, discrete spectrum, essential spectrum, resolvent set and spectral measure of $A$ will be denote by
$\Dom(A)$, $\sigma(A)$, $\sigma_{\dis}(A)$, $\sigma_{\ess}(A)$, $\rho(A)$, and $E$ respectively. Unless otherwise stated we shall
assume that $A$ is bounded. The most commonly used technique for attempting to approximate the spectrum of a linear operator is the
finite-section method: we choose a finite-dimensional subspace $\cL$ with corresponding
orthogonal projection $P$, and calculate the eigenvalues of $PA|_\cL$. If $(\cL_n)_{n\in\mathbb{N}}$ is a sequence of finite-dimensional subspaces
such that the corresponding orthogonal projections $(P_n)$ converge strongly to the identity operator, then we write $(\cL_n)\in\Lambda$.
For such a sequence we define the following limit set
\begin{displaymath}
\lim_{n\to\infty}\sigma(P_nA|_{\cL_n}) = \{z\in\mathbb{C}:\textrm{ there exist }z_n\in\sigma(P_nA|_{\cL_n})\textrm{ with }z_n\to z\}.
\end{displaymath}
For a bounded self-adjoint operator we have
\begin{displaymath}
\lim_{n\to\infty}\sigma(P_nA|_{\cL_n}) \supseteq\sigma(A)\quad\textrm{and}\quad\lim_{n\to\infty}\sigma(P_nA|_{\cL_n}) \subseteq \sigma(A)\cup\conv(\sigma_{\ess}(A))
\end{displaymath}
where $\conv$ denotes the closed convex hull (see for example \cite[Theorem 6.1]{shar}). That the limit set contains $\sigma(A)$
is encouraging; however, this containment can be strict. We say that a $z\in\rho(A)$ is a point of spectral pollution for $(\cL_n)\in\Lambda$ if $z$ belongs to the limit set.
This constitutes a serious problem since spectral pollution can occur anywhere inside a gap in the essential spectrum
(see \cite[Section 2.1]{bost}, \cite[Theorem 2.1]{lesh}, \cite[Theorem 6.1]{shar}). Consequently,
the finite-section method often fails to identify eigenvalues in such gaps (see for example \cite{boff,bodu,daug,rapa}).
The situation for normal operators can be far worse as the following example shows.

\begin{example}\label{ex1}Let $A$ be the bilateral shift operator acting on $\mathcal{H}:=\ell_2(\mathbb{Z})$:
\begin{displaymath}
\textrm{if}\quad(a_k)_{k=-\infty}^\infty\in\ell_2(\mathbb{Z})\quad\textrm{then}\quad
A(a_k)_{k=-\infty}^\infty = (a_{k-1})_{k=-\infty}^\infty.
\end{displaymath}
The operator $A$ is unitary and
$\sigma(A) = \{z\in\mathbb{C}:\vert z\vert = 1\}$. Let $\cL_n = \Span\{e_k\}_{k=-n}^n$ where $e_k$ is the sequence which has zeros
in all slots except the $k^{th}$ which has a $1$. The $\{e_k\}_{k=-\infty}^\infty$ form an orthonormal basis for
$\ell_2(\mathbb{Z})$ and hence $(\cL_n)\in\Lambda$. We obtain $\sigma(P_nA|_{\cL_n}) = \{0\}$ for all $n\in\mathbb{N}$. Therefore the limit set does not even intersect $\sigma(A)$.
\end{example}

For self-adjoint operators there are very few techniques available for avoiding pollution. Notable amongst these is the is the second order relative spectrum (see \cite{dapl,marc,zim}). To apply this method we must solve the quadratic eigenvalue problem $P(A - zI)^2\phi = 0$ for $\phi\in\cL\backslash\{0\}$.
For a self-adjoint $A$, we denote the solutions to this eigenvalue problem by $\Spec_2(A,\cL)$. We have the following useful property: if $z\in \Spec_2(A,\cL)$, then
\begin{equation}\label{lower}
\big[\Re z - \vert\Im z\vert,\Re z+\vert\Im z\vert\big]\cap\sigma(A)\ne\varnothing
\end{equation}
(see \cite[Corollary 4.2]{shar}). This property can often be significantly improved: if $\sigma(A)\cap(a,b) = \{\lambda\}$,
$\mathbb{D}(a,b)$ is the open disc with center
$(a+b)/2$ and radius $(b-a)/2$, and $z\in\Spec_2(A,\cL)\cap\mathbb{D}(a,b)$, then
\begin{equation}\label{high}
\left[\Re z - \frac{\vert\Im z\vert^2}{b - \Re z},\Re z + \frac{\vert\Im z\vert^2}{\Re z - a}\right]\cap\sigma(A) = \{\lambda\}
\end{equation}
(see \cite[Theorem 2.1 and Remark 2.2]{str}, see also \cite[Corollary 2.6]{bole} and \cite{kato0}). The
method is also guaranteed to converge to the discrete spectrum of a self-adjoint operator:
if $(a,b)\cap\sigma(A)\subset\sigma_{\dis}(A)$, then we have
\begin{equation}\label{bostth}
\Big(\lim_{n\to\infty}\Spec_2(A,\cL_n)\Big)\cap\mathbb{D}(a,b) =  \sigma_{\dis}(A)\cap(a,b)\quad\textrm{for all}\quad(\cL_n)\in\Lambda
\end{equation}
(see \cite[Corollary 8]{bost}, also \cite[Theorem 1]{bo}). The method can be traced back to \cite{da} and has been successfully applied to
self-adjoint operators from solid state physics \cite{bole},
relativistic quantum mechanics \cite{bobo}, Stokes systems \cite{lesh}, and magnetohydrodynamics \cite{str}. Applying this method to
normal operators does not seem encouraging as the following example shows.

\begin{example}\label{ex2}Let $A$, $\mathcal{H}$ and $\cL_n$ be as in Example \ref{ex1}. We find that zero is the only solution
to the quadratic eigenvalue problem  $P_n(A - zI)^2|_{\cL_n}$ all $n\in\mathbb{N}$.
\end{example}

For non-self-adjoint operators, the solutions to the quadratic eigenvalue problem $P(A - zI)^2|_\cL$ have received very little
attention (see \cite{shar}). The failure in Example \ref{ex2} of the solutions to converge to any points in the spectrum of the operator is highly unsatisfactory and motivates the following definition.

\begin{definition}\label{eugenelevitin}
Let $A$ be a (possibly unbounded) normal operator and let $\cL$ be a subspace of $\Dom(A)$. The second order spectrum of $A$ relative to $\cL$ is the set
\begin{displaymath}
\Spec_2(A,\cL):=\{z\in\mathbb{C}:\textrm{there exists a }\phi\in\cL\backslash\{0\}\textrm{ with }
\langle(A - z)\phi,(A-\overline{z})\psi\rangle = 0\textrm{ for all }\psi\in\cL\}.
\end{displaymath}
\end{definition}
This is a generalisation of the definition which appears in \cite{lesh} for self-adjoint operators. If $A$ is bounded, then the second order spectrum of $A$ relative to $\cL$ is precisely the solutions to the quadratic eigenvalue
problem $P(A-zI)(A^*-zI)\phi = 0$ for $\phi\in\cL\backslash\{0\}$.  We note that for non-self-adjoint operators this definition differs from that which appears
in \cite{shar}.

In Section 2 we discuss some geometric issues which will cast light on the geometry of the second order relative spectrum.
In Section 3 we linearise the quadratic eigenvalue problem which arises from Definition \ref{eugenelevitin}. By doing this
we are better able to understand how and why the method converges to both eigenvalues and eigenspaces. In Section 4 we obtain
convergence estimates for eigenvalues and eigenspaces. We use the following notion of the gap between two subspaces $\cL,\mathcal{M}\subset\mathcal{H}$
\begin{displaymath}
\delta(\cL,\mathcal{M}) = \sup_{\psi\in\cL,~\Vert\psi\Vert=1}\dist[\psi,\mathcal{M}]\quad\textrm{and}\quad\hat{\delta}(\cL,\mathcal{M}) = \max\{\delta(\cL,\mathcal{M}),\delta(\mathcal{M},\cL)\}
\end{displaymath}
(see for example \cite[Section IV.2.1]{katopert}). For eigenvalues $z_1,\dots,z_m\in\sigma_{\dis}(A)$
 the corresponding linear hull of eigenspaces will be denoted $\cL(\{z_1,\dots,z_m\})$.
The main result is Theorem \ref{lim3} which applied to a self-adjoint operator $A$ with $z\in\sigma_{\dis}(A)$
yields $\dist[z,\Spec_2(A,\cL_n)] = \mathcal{O}(\delta(\cL(\{z\}),\cL_n))$. This improves upon
the previous estimate $\dist[z,\Spec_2(A,\cL_n)] = \mathcal{O}(\delta(\cL(\{z\}),\cL_n)^{\frac{1}{2}})$ (see \cite{bo2,bost}).
For a $z\in\sigma_{\dis}(A)$ we therefore have a sequence $z_n\in\Spec_2(A,\cL_n)$
with $\vert z_n - z\vert = \mathcal{O}(\delta(\cL(\{z\}),\cL_n))$. If we combine this with property \eqref{high} we obtain
$\vert \Re z_n - z\vert = \mathcal{O}(\delta(\cL(\{z\}),\cL_n)^2)$. This is the same order of convergence to
an arbitrary member of $\sigma_{\dis}(A)$ that the finite-section method achieves (see for example \cite{chat}).
Moreover, we find approximate eigenspaces $\mathcal{M}_n(\{z\})$ with
$\hat{\delta}(\cL(\{z\}),\mathcal{M}_n(\{z\})) = \mathcal{O}(\delta(\cL(\{z\}),\cL_n))$. Again, this is the same order of convergence
that the finite-section method achieves for eigenspaces. However, due to spectral pollution, this
convergence to eigenspaces in the finite-section method applies only to those eigenvalues outside
$\conv(\sigma_{\ess}(A))$. The section includes a simple example where the convergence rates are achieved.
In Section 5 we show that the second order spectrum provides enclosures for eigenvalues of normal operators.
The final section extends the results to unbounded operators.

\section{Geometric Preliminaries}

Throughout this section $\Sigma$ will be an arbitrary compact subset of $\mathbb{C}$. For an $\varepsilon>0$ and $z\in\mathbb{C}$,
we introduce the following sets
\begin{align*}
[\Sigma]_\varepsilon := \{z&\in\mathbb{C}:\dist[z,\Sigma]\le\varepsilon\},\quad
\Sigma_z := \{(\lambda - z)(\overline{\lambda} - z):\lambda\in\Sigma\},\\
&\quad\textrm{and}\quad\mathcal{Q}(\Sigma) := \{z\in\mathbb{C}:0\in\conv(\Sigma_z)\}.
\end{align*}
We study these sets because they will give us an insight into the geometry of the
second order relative spectrum. The sets are similar to $\Sigma_z^2 := \{(\lambda - z)^2:\lambda\in\Sigma\}$ and
$\mathcal{Q}_2(\Sigma) := \{z\in\mathbb{C}:0\in\textrm{conv}(\Sigma_z^2)\}$ which were introduced in \cite{shar}. Our reason for
studying $\Sigma_z$ and $\mathcal{Q}(\Sigma)$ - rather than  $\Sigma_z^2$ and $\mathcal{Q}_2(\Sigma)$ - is that our definition
of the second order relative spectrum differs from that used in \cite{shar}.

The assertions of the following lemma follow immediately from the definition of $\mathcal{Q}(\Sigma)$.

\begin{lemma}\label{cor0}
Let $\Sigma$ be a compact subset of $\mathbb{C}$, then $\Sigma\subset\mathcal{Q}(\Sigma)$, $\Sigma\cap\mathbb{R}=\mathcal{Q}(\Sigma)\cap\mathbb{R}$, and $z\in\mathcal{Q}(\Sigma)$ if and only if $\overline{z}\in\mathcal{Q}(\Sigma)$.
\end{lemma}

Let $\lambda_1,\lambda_2\in\mathbb{C}$, then
$\mathcal{Q}(\{\lambda_j\}) = \{\lambda_j,\overline{\lambda}_j\}$ and $z\in\mathcal{Q}(\{\lambda_1,\lambda_2\})$ if and only if for some $t\in[0,1]$
we have
\begin{equation}\label{curve}
t(\vert\lambda_1\vert^2 - 2z\Re \lambda_1 + z^2) + (1-t)(\vert\lambda_2\vert^2 - 2z\Re \lambda_2 + z^2) = 0.
\end{equation}
Consider the curves $\gamma(\lambda_1,\lambda_2)^\pm(\cdot):[0,1]\to\mathbb{C}$ defined by
\begin{align*}
\gamma(\lambda_1,\lambda_2)^\pm(t) &:=  t\Re\lambda_1 + (1-t)\Re\lambda_2 \pm \sqrt{(t\Re\lambda_1 + (1-t)\Re\lambda_2)^2 -
t\vert\lambda_1\vert^2 - (1-t)\vert\lambda_2\vert^2}\\
&\:=  t\Re\lambda_1 + (1-t)\Re\lambda_2 \pm i\sqrt{t(1-t)(\Re\lambda_1- \Re\lambda_2)^2 + t(\Im\lambda_1)^2 +
(1-t)(\Im\lambda_2)^2},
\end{align*}
and set $\gamma(\lambda_1,\lambda_2)^\pm = \{z\in\mathbb{C}:z=\gamma(\lambda_1,\lambda_2)^\pm(t)\textrm{ for some }t\in[0,1]\}$.
It follows from \eqref{curve} that $\mathcal{Q}(\{\lambda_1,\lambda_2\}) = \gamma(\lambda_1,\lambda_2)^+\cup\gamma(\lambda_1,\lambda_2)^-$. If $\Re\lambda_1 = \Re\lambda_2$ then clearly $\gamma(\lambda_1,\lambda_2)^+$ is the straight line between $\Re\lambda_1 + i\vert\Im\lambda_1\vert$ and $\Re\lambda_2 + i\vert\Im\lambda_2\vert$, and $\gamma(\lambda_1,\lambda_2)^-$ is the straight line between $\Re\lambda_1 - i\vert\Im\lambda_1\vert$ and $\Re\lambda_2 - i\vert\Im\lambda_2\vert$. If $\Re\lambda_1 \ne \Re\lambda_2$ then for any $t\in[0,1]$ we find that
\begin{displaymath}
\left\vert\gamma(\lambda_1,\lambda_2)^\pm(t)-\frac{\vert\lambda_2\vert^2-\vert\lambda_1\vert^2}{2(\Re\lambda_2 - \Re\lambda_1)}\right\vert^2 = \frac{(\vert\lambda_2\vert^2-\vert\lambda_1\vert^2)^2}{4(\Re\lambda_2 - \Re\lambda_1)^2} -\Re\lambda_2\frac{\vert\lambda_2\vert^2-\vert\lambda_1\vert^2}{(\Re\lambda_2 - \Re\lambda_1)}  + \vert\lambda_2\vert^2,
\end{displaymath}
therefore $\gamma(\lambda_1,\lambda_2)^\pm$ are arcs of the circle with center $c$ and radius $r$ where
\begin{equation}\label{circles}
c=\frac{\vert\lambda_2\vert^2-\vert\lambda_1\vert^2}{2(\Re\lambda_2 - \Re\lambda_1)}\quad\textrm{and}\quad r^2=\frac{(\vert\lambda_2\vert^2-\vert\lambda_1\vert^2)^2}{4(\Re\lambda_2 - \Re\lambda_1)^2} -\Re\lambda_2\frac{\vert\lambda_2\vert^2-\vert\lambda_1\vert^2}{(\Re\lambda_2 - \Re\lambda_1)}  + \vert\lambda_2\vert^2.
\end{equation}
If $L$ is the line segment between $\Re\lambda_1 + i\vert\Im\lambda_1\vert$ and $\Re\lambda_2 + i\vert\Im\lambda_2\vert$, then the real number  $c$ is the point where the perpendicular bisector of $L$ meets the real line. The radius $r$ is then the distance between $c$ and $\Re\lambda_1 \pm i\vert\Im\lambda_1\vert$ (and $\Re\lambda_2 \pm i\vert\Im\lambda_2\vert$). Now let $\lambda_1,\lambda_2,\lambda_3\in\Sigma$ with
$\Re\lambda_1\le\Re\lambda_2\le\Re\lambda_3$. It follows that either $\lambda_2\in\gamma(\lambda_1,\lambda_3)^-\cup\gamma(\lambda_1,\lambda_3)^+$ and
$\gamma(\lambda_1,\lambda_3)^\pm= \gamma(\lambda_1,\lambda_2)^\pm\cup\gamma(\lambda_2,\lambda_3)^\pm$, or $\lambda_2\notin\gamma(\lambda_1,\lambda_3)^-\cup\gamma(\lambda_1,\lambda_3)^+$ and
\begin{align*}
\gamma(\lambda_1,\lambda_3)^+\cap \big(\gamma(\lambda_1,\lambda_2)^+\cup\gamma(\lambda_2,\lambda_3)^+\big)& =
\big\{\Re\lambda_1+i\vert\Im\lambda_1\vert,\Re\lambda_3+i\vert\Im\lambda_3\vert\big\},\\
\gamma(\lambda_1,\lambda_3)^-\cap \big(\gamma(\lambda_1,\lambda_2)^-\cup\gamma(\lambda_2,\lambda_3)^-\big) &=
\big\{\Re\lambda_1-i\vert\Im\lambda_1\vert,\Re\lambda_3-i\vert\Im\lambda_3\vert\big\}.
\end{align*}
Let $q(\lambda_1,\lambda_2,\lambda_3)^\pm:=\gamma(\lambda_1,\lambda_2)^\pm\cup\gamma(\lambda_2,\lambda_3)^\pm\cup
\gamma(\lambda_1,\lambda_3)^\pm$, then $q(\lambda_1,\lambda_2,\lambda_3)^+$ and $q(\lambda_1,\lambda_2,\lambda_3)^-$
are simple closed curves. We denote the closed interiors of these curves by
$\overline{\intr}(q(\lambda_1,\lambda_2,\lambda_3)^\pm)$. Figures 1-3 show $\overline{\intr}(q(\lambda_1,\lambda_2,\lambda_3)^\pm)$
for three different situations.

\begin{figure}[h!]
\centering
\includegraphics[scale=.8]{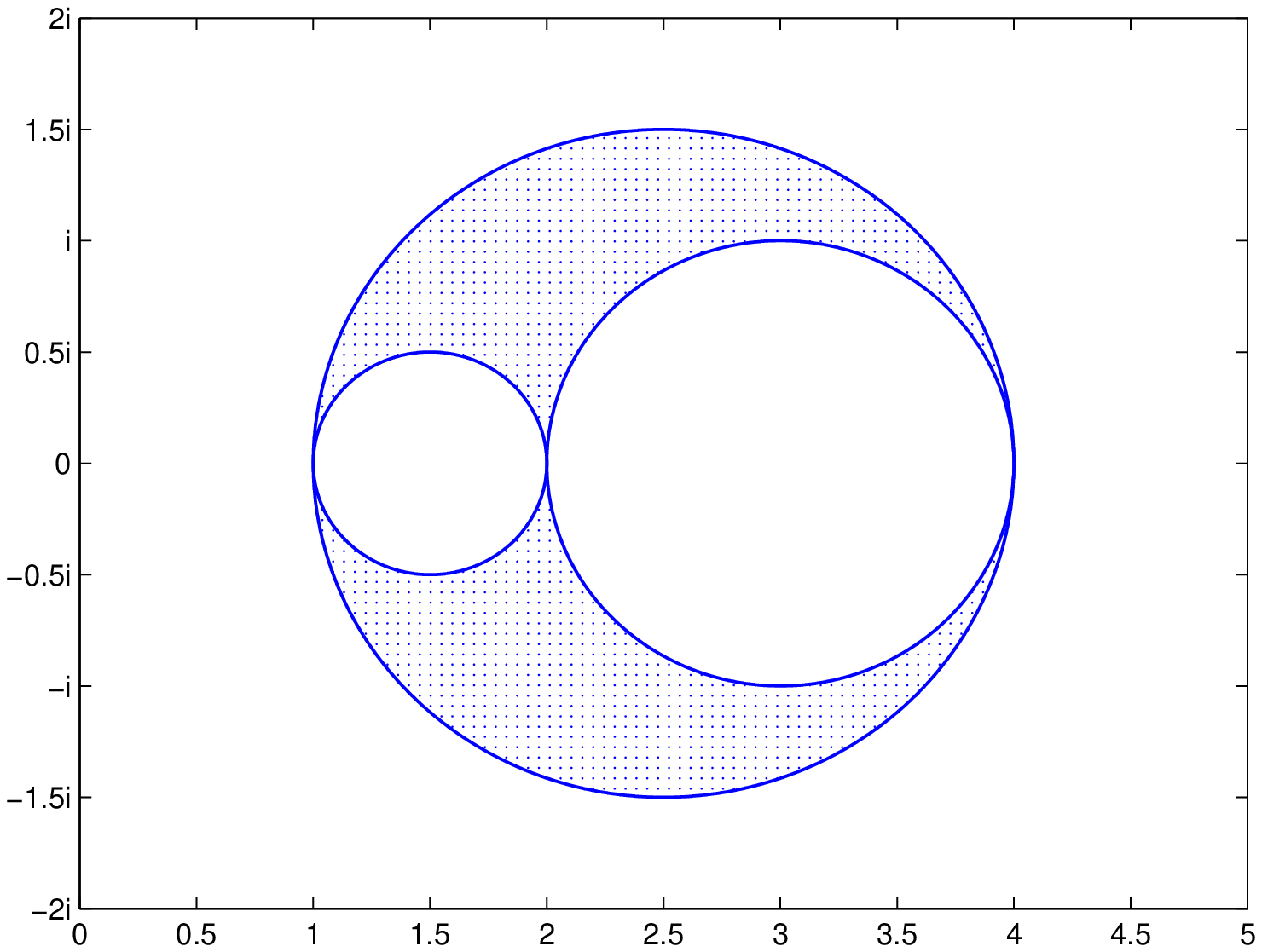}
\caption{The shaded region above and below the real line are $\overline{\intr}(q(1,2,4)^+)$ and $\overline{\intr}(q(1,2,4)^-)$, respectively. The regions are enclosed by the arcs $\gamma(1,2)^\pm$, $\gamma(1,4)^\pm$ and $\gamma(2,4)^\pm$.}
\end{figure}

\begin{figure}[h!]
\centering
\includegraphics[scale=.8]{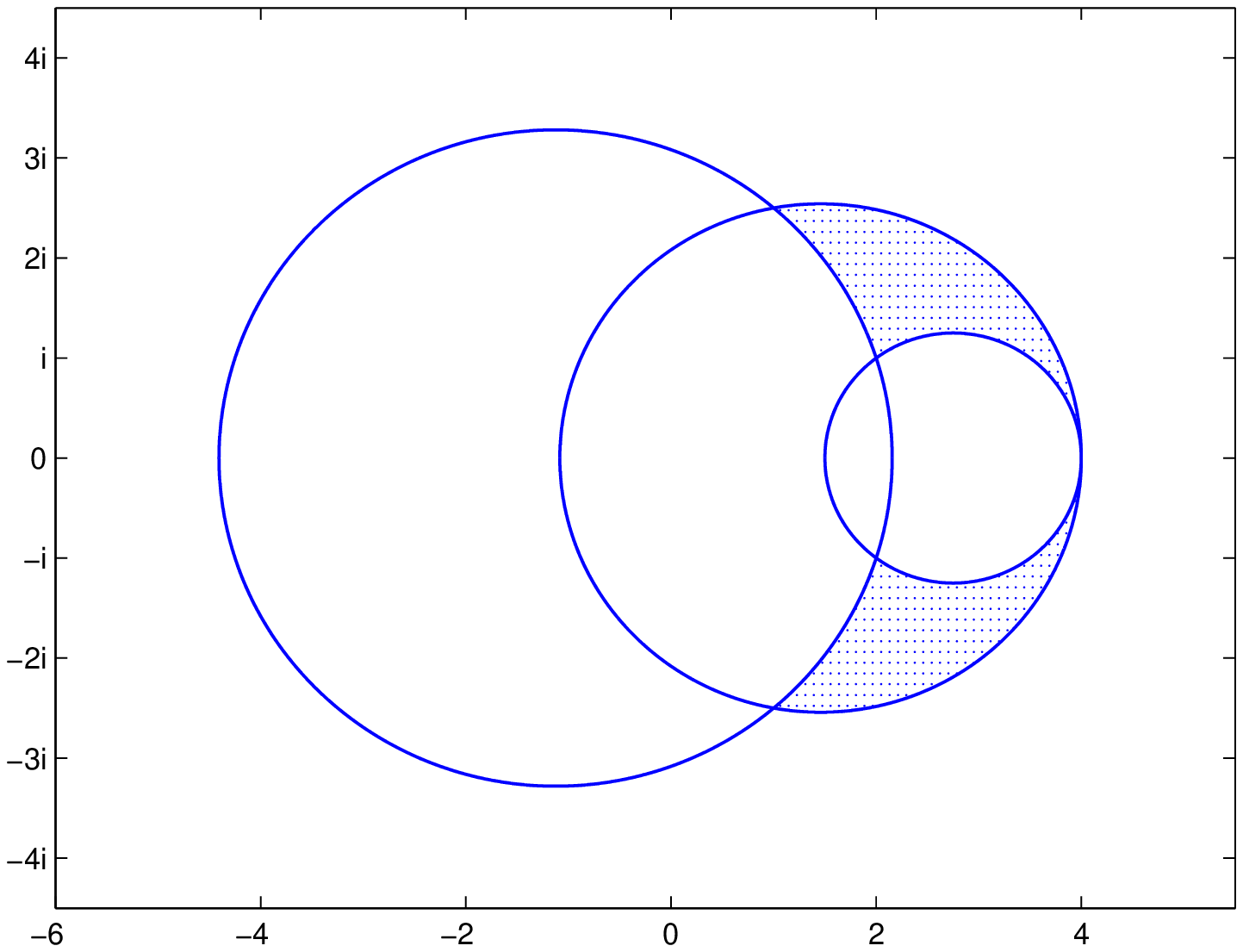}
\caption{The shaded region above and below the real line are $\overline{\intr}(q(1+2.5i,2+i,4)^+)$ and $\overline{\intr}(q(1+2.5i,2+i,4)^-)$, respectively. The regions are enclosed by the arcs $\gamma(1+2.5i,2+i)^\pm$, $\gamma(1+2.5i,4)^\pm$ and $\gamma(2+i,4)^\pm$.}
\end{figure}

\begin{figure}[h!]
\centering
\includegraphics[scale=.8]{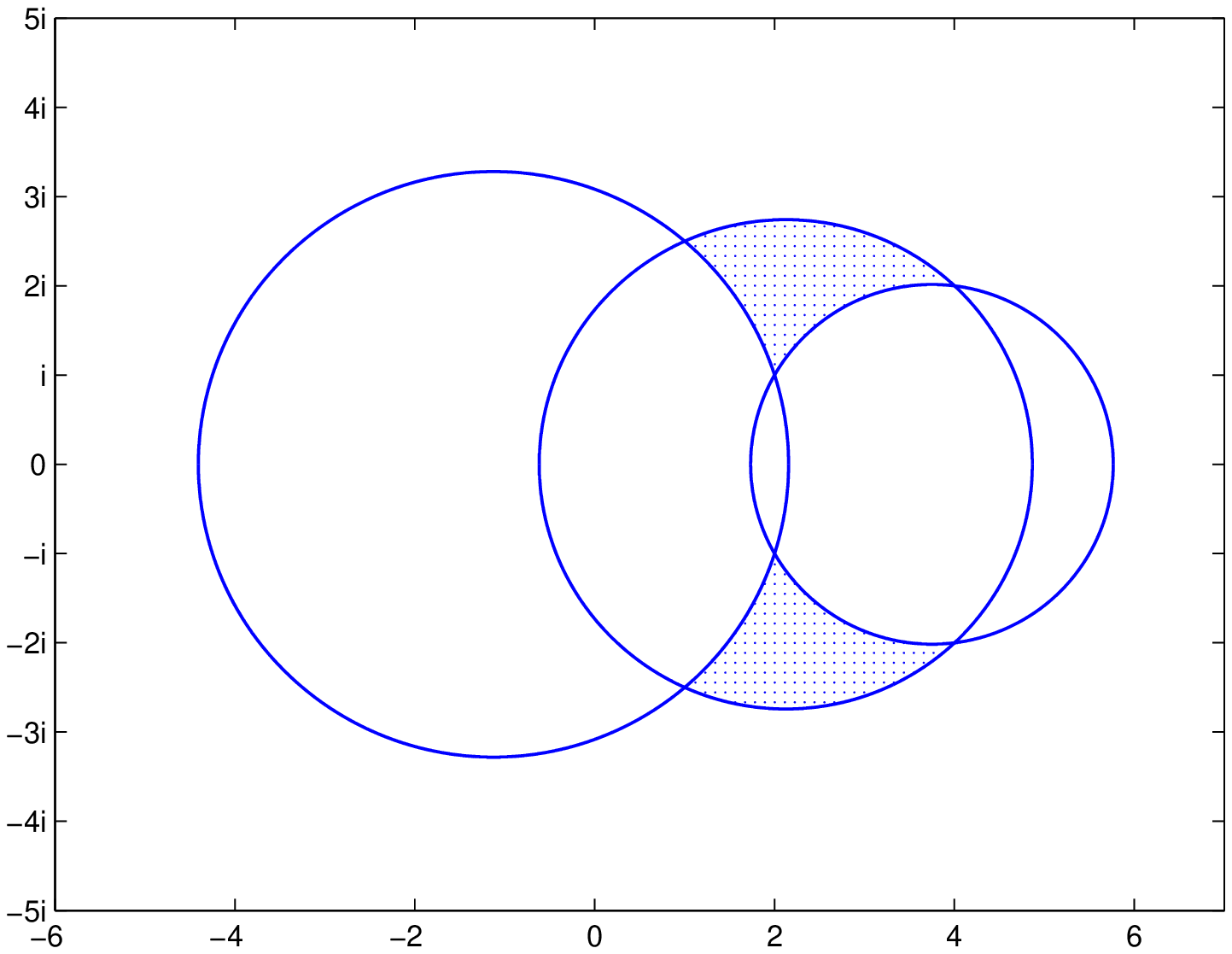}
\caption{The shaded region above and below the real line are $\overline{\intr}(q(1+2.5i,2+i,4+2i)^+)$ and $\overline{\intr}(q(1+2.5i,2+i,4+2i)^-)$, respectively. The regions are enclosed by the arcs $\gamma(1+2.5i,2+i)^\pm$, $\gamma(1+2.5i,4+2i)^\pm$ and $\gamma(2+i,4)^\pm$.}
\end{figure}

\begin{theorem}\label{thm1}
Let $\Sigma$ be a compact subset of $\mathbb{C}$, then
\begin{equation}\label{Q}
\mathcal{Q}(\Sigma) = \Big\{z\in\mathbb{C}:~\textrm{there exist~}\lambda_1,\lambda_2,\lambda_3\in\Sigma\textrm{~such that~}
z\in\overline{\intr}(q(\lambda_1,\lambda_2,\lambda_3)^+)\cup\overline{\intr}(q(\lambda_1,\lambda_2,\lambda_3)^-)\Big\}.
\end{equation}
\end{theorem}
\begin{proof}
Let $z$ belong to the left hand side of \eqref{Q} and without loss of generality suppose that $\Im z\ge 0$. It follows from the definition of $\mathcal{Q}(\Sigma)$ that there exist $\lambda_1,\lambda_2,\lambda_3\in\Sigma$  where $\Re\lambda_1\le\Re\lambda_2\le\Re\lambda_3$, such that
\begin{equation}\label{prebary}
0\in\conv(\vert\lambda_1\vert^2 - 2z\Re\lambda_1+z^2,\vert\lambda_2\vert^2 - 2z\Re\lambda_2+z^2,\vert\lambda_3\vert^2 - 2z\Re\lambda_3+z^2).
\end{equation}
From \eqref{prebary} it follows that for some $\hat{s},\hat{t}\in[0,1]$ we have
\begin{displaymath}
0 = \hat{t}(\vert\lambda_1\vert^2 - 2z\Re\lambda_1+z^2)+(1-\hat{t})\hat{s}(\vert\lambda_2\vert^2 - 2z\Re\lambda_2+z^2)+(1-\hat{t})(1-\hat{s})(\vert\lambda_3\vert^2 - 2z\Re\lambda_3+z^2),
\end{displaymath}
from which we obtain
\begin{align*}
\Re z &= \hat{t}\Re\lambda_1 + (1-\hat{t})\hat{s}\Re\lambda_2 + (1-\hat{t})(1-\hat{s})\Re\lambda_3\\
\Im z &= \sqrt\big(\hat{s}\hat{t}(1-\hat{t})(\Re\lambda_1 - \Re\lambda_2)^2 + \hat{t}(1-\hat{t})(1-\hat{s})(\Re\lambda_1-\Re\lambda_3)^2\\
&~~~~~~ +(1-\hat{t})^2\hat{s}(1-\hat{s})(\Re\lambda_2 - \Re\lambda_3)^2+\hat{t}(\Im\lambda_1)^2
+ (1-\hat{t})\hat{s}(\Im\lambda_2)^2 + (1-\hat{t})(1-\hat{s})(\Im\lambda_3)^2\big).
\end{align*}
For some $t_0\in[0,1]$ we have $\Re z = t_0\Re\lambda_1 + (1-t_0)\Re\lambda_3$. We assume that $\Re\lambda_1<\Re\lambda_2<\Re\lambda_3$, the case where $\Re\lambda_1=\Re\lambda_2$ and/or $\Re\lambda_2=\Re\lambda_3$ being treated similarly. If $t_0=1$, then $\hat{t} = 1$ and $\hat{s}=0$, so that $z = \lambda_1\in\gamma(\lambda_1,\lambda_2)^+\subset \overline{\intr}(q(\lambda_1,\lambda_2,\lambda_3)^+)$. Similarly, if $t_0=0$ we have $z = \lambda_3 \in \overline{\intr}(q(\lambda_1,\lambda_2,\lambda_3)^+)$. Suppose now that $t_0\in(0,1)$. We assume that $\Re z\in[\Re\lambda_1,\Re\lambda_2]$, the case were $\Re z\in[\Re\lambda_2,\Re\lambda_3]$ being treated similarly. Define
\begin{displaymath}
s(t) = \frac{(t_0-t)(\Re\lambda_1-\Re\lambda_3)}{(1-t)(\Re\lambda_2-\Re\lambda_3)}\quad\textrm{and}\quad t_1 = \frac{\Re z - \Re\lambda_2}{\Re\lambda_1 - \Re\lambda_2},
\end{displaymath}
and consider $z(t)\in\mathbb{C}$ where
\begin{align*}
\Re z(t) &=  t\Re\lambda_1 + (1-t)s(t)\Re\lambda_2 + (1-t)(1-s(t))\Re\lambda_3\\
\Im z(t) &= \sqrt\big(s(t)t(1-t)(\Re\lambda_1 - \Re\lambda_2)^2 + t(1-t)(1-s(t))(\Re\lambda_1-\Re\lambda_3)^2\\
&~~~~~~~~~~~ +(1-t)^2s(t)(1-s(t))(\Re\lambda_2 - \Re\lambda_3)^2 +t(\Im\lambda_1)^2\\
&~~~~~~~~~~~~~~~~~ + (1-t)s(t)(\Im\lambda_2)^2+ (1-t)(1-s(t))(\Im\lambda_3)^2\big).
\end{align*}
It is straightforward to verify that $t_1\le t_0$, $z(t_1)\in\gamma(\lambda_1,\lambda_2)^+$, $z(t_0)\in\gamma(\lambda_1,\lambda_3)^+$, $\Re z(t) = \Re z$ and $0\le s(t)\le 1$ for all $t\in[t_1,t_0]$. Also, we have
\begin{displaymath}
\frac{\textrm{d}}{\textrm{dt}}\Im z(t) = \bigg(-\vert\lambda_1\vert^2 + \frac{\Re\lambda_1 - \Re\lambda_3}{\Re\lambda_2 - \Re\lambda_3}\vert\lambda_2\vert^2 + \frac{\Re\lambda_2 - \Re\lambda_1}{\Re\lambda_2 - \Re\lambda_3}\vert\lambda_3\vert^2\Bigg)\bigg/\Im z(t).
\end{displaymath}
In particular we note that for $t\in(t_1,t_0)$ the sign of the derivative does not change. It follows that $z(t)\in\overline{\intr}(q(\lambda_1,\lambda_2,\lambda_3)^+)$ for all $t\in[t_1,t_0]$. It will now suffice to show that $\hat{t}\in[t_1,t_0]$. If $\hat{t}>t_0$, then
\begin{align*}
\Re z &= \hat{t}\Re\lambda_1 + (1-\hat{t})(\hat{s}\Re\lambda_2 + (1-\hat{s})\Re\lambda_3)\\
&< t_0\Re\lambda_1 + (1-t_0)(\hat{s}\Re\lambda_2 + (1-\hat{s})\Re\lambda_3)\\
&<t_0\Re\lambda_1 + (1-t_0)\Re\lambda_3
\end{align*}
which is a contradiction since the right hand side equals $\Re z$. If $\hat{t}<t_1$, then
\begin{align*}
\Re z &= \hat{t}\Re\lambda_1 + (1-\hat{t})(\hat{s}\Re\lambda_2 + (1-\hat{s})\Re\lambda_3)\\
&>t_1\Re\lambda_1 + (1-t_1)(\hat{s}\Re\lambda_2 + (1-\hat{s})\Re\lambda_3)\\
&>t_1\Re\lambda_1 + (1-t_1)\Re\lambda_2,
\end{align*}
which is a contradiction since the right hand side equals $\Re z$ (since $s(t_1)=1$). We deduce that $\mathcal{Q}(\Sigma)$ is contained in the right hand side of \eqref{Q}.

Now let $z$ belong to the right and side of \eqref{Q}. Without loss of generality we suppose that $\Im z\ge 0$.
It follows from the above that for some $s,t\in[0,1]$ we have
\begin{align*}
\Re z &= t\Re\lambda_1 + (1-t)s\Re\lambda_2 + (1-t)(1-s)\Re\lambda_3\\
\Im z &= \sqrt\big(st(1-t)(\Re\lambda_1 - \Re\lambda_2)^2 + t(1-t)(1-s)(\Re\lambda_1-\Re\lambda_3)^2\\
&~~~~~~ +(1-t)^2s(1-s)(\Re\lambda_2 - \Re\lambda_3)^2+t(\Im\lambda_1)^2
+ (1-t)s(\Im\lambda_2)^2 + (1-t)(1-s)(\Im\lambda_3)^2\big).
\end{align*}
Therefore
\begin{displaymath}
0 = t(\vert\lambda_1\vert^2 - 2z\Re\lambda_1+z^2)+(1-t)s(\vert\lambda_2\vert^2 - 2z\Re\lambda_2+z^2)+(1-t)(1-s)(\vert\lambda_3\vert^2 - 2z\Re\lambda_3+z^2),
\end{displaymath}
so that $0\in\conv(\vert\lambda_1\vert^2 - 2z\Re\lambda_1+z^2,\vert\lambda_2\vert^2 - 2z\Re\lambda_2+z^2,\vert\lambda_3\vert^2 - 2z\Re\lambda_3+z^2)$,
and $z\in\mathcal{Q}(\Sigma)$.
\end{proof}

\begin{corollary}\label{cor2}
Let $\varepsilon>0$,  then $\mathcal{Q}([\Sigma]_\varepsilon)=[\mathcal{Q}(\Sigma)]_\varepsilon$ and
$\dist(0,\conv(\Sigma_z))\ge\varepsilon^2$ for any $z\notin[\mathcal{Q}(\Sigma)]_\varepsilon$.
\end{corollary}
\begin{proof}
Let $\lambda_1,\lambda_2,\lambda_3\in\Sigma$. The region $\overline{\intr}(q(\lambda_1,\lambda_2,\lambda_3)^\pm)$ is that enclosed by $\gamma(\lambda_1,\lambda_2)^\pm$, $\gamma(\lambda_2,\lambda_3)^\pm$ and $\gamma(\lambda_1,\lambda_3)^\pm$,
where the $\gamma(\lambda_i,\lambda_j)^\pm$ are either straight lines or arcs of circles centered on the real line (see \eqref{circles}).
Evidently, $\gamma(\hat{\lambda}_i,\hat{\lambda}_j)^\pm\subset[\gamma(\lambda_i,\lambda_j)^\pm]_\varepsilon$ for any
$\hat{\lambda}_j\in\{z:\vert\lambda_j-z\vert\le\varepsilon\}$, therefore
$\mathcal{Q}([\Sigma]_\varepsilon)\subset[\mathcal{Q}(\Sigma)]_\varepsilon$ follows from Theorem \ref{thm1}.

For any $z\in[\overline{\intr}(q(\lambda_1,\lambda_2,\lambda_3)^\pm)]_\varepsilon$ we have either:
$z\in\overline{\intr}(q(\lambda_1,\lambda_2,\lambda_3)^\pm)$, or $z\in\gamma(\hat{\lambda}_i,\hat\lambda_j)^\pm$ for
$\hat{\lambda}_i\in[\lambda_i]_\varepsilon$ and $\hat{\lambda}_j\in[\lambda_j]_\varepsilon$. In either case we have
$z\in\overline{\intr}(q(\hat{\lambda}_1,\hat{\lambda}_2,\hat{\lambda}_3)^\pm)$ for some
$\hat{\lambda}_j\in[\lambda_j]_\varepsilon$, therefore
$\mathcal{Q}([\Sigma]_\varepsilon)\supset[\mathcal{Q}(\Sigma)]_\varepsilon$ follows from Theorem \ref{thm1}.

For the last assertion we suppose that $z\notin[\mathcal{Q}(\Sigma)]_\varepsilon$ and $\dist(0,\conv(\Sigma_z))<\varepsilon^2$. Then for some $\lambda_1,\lambda_2\in\Sigma$ and $t\in[0,1]$, we have
\begin{displaymath}
\varepsilon^2 > \vert t(\lambda_1-z)(\overline{\lambda}_1 - z) - (1-t)(\lambda_2-z)(\overline{\lambda}_2 - z)\vert\\
=\vert z - \gamma^+(\lambda_1,\lambda_2)(t)\vert\vert z - \gamma^-(\lambda_1,\lambda_2)(t)\vert.
\end{displaymath}
Since $\gamma^\pm(\lambda_1,\lambda_2)(t)\in\mathcal{Q}(\Sigma)$ we obtain a contradiction.
\end{proof}

\section{Linearisation}

A quadratic eigenvalue problem can be expressed as a linear eigenvalue problem for a block operator matrix, and for this reason
we consider the following operator
\begin{displaymath}
T := \left(
\begin{array}{cc}
A+A^* & -A^*A\\
I & 0
\end{array} \right):\cH\oplus\cH\to\cH\oplus\cH.
\end{displaymath}
We note that a direct calculation verifies that for any non-zero $w\in\rho(T)$ we have
\begin{equation}\label{mult0}
(T-w)^{-1} = \left(
\begin{array}{cc}
-w(A^*-wI)^{-1}(A-wI)^{-1}& (A^*-wI)^{-1}(A-wI)^{-1}A^*A\\
-(A^*-wI)^{-1}(A-wI)^{-1} & -w^{-1} + w^{-1}(A^*-wI)^{-1}(A-wI)^{-1}A^*A
\end{array} \right),
\end{equation}
For an eigenvalue $z\in\sigma_{\dis}(T)$ the corresponding spectral subspace
will be denoted by $\mathcal{M}(\{z\})$, and recall that for an eigenvalue $z\in\sigma_{\dis}(A)$
the corresponding spectral subspace
is denoted by $\cL(\{z\})$. In the statement of the following lemma we consider a $z\in\sigma_{\dis}(A)\cup\sigma_{\dis}(A^*)$,
together with the eigenspaces $\cL(\{z\})$ and $\cL(\{\overline{z}\})$. The latter is therefore the eigenspace associate to $A$ and
$\overline{z}$, so that $\cL(\{\overline{z}\})$ contains non-zero vectors if and only if $\overline{z}\in\sigma_{\dis}(A)$.

\begin{lemma}\label{eigspaces}
We have $\sigma(T) = \sigma(A)\cup\sigma(A^*)$. If $z\in\sigma_{\dis}(A)\cup\sigma_{\dis}(A^*)$ with
$\cL(\{z\})=\Span\{\phi_1,\dots,\phi_k\}$ and $\cL(\{\overline{z}\})=\Span\{\phi_{k+1},\dots,\phi_{k+m}\}$,
then $z,\overline{z}\in\sigma_{\dis}(T)$ and
\begin{align}
\mathcal{M}(\{z\}) &= \Span\left\{\left(
\begin{array}{c}
z\phi_1\\
\phi_1
\end{array} \right),\dots,\left(
\begin{array}{c}
z\phi_{k+m}\\
\phi_{k+m}
\end{array} \right)\right\}\quad\textrm{if}\quad z\notin\mathbb{R}\label{imaginary}\\
\mathcal{M}(\{z\}) &= \Span\left\{\left(
\begin{array}{c}
0\\
\phi_1
\end{array} \right),\left(
\begin{array}{c}
\phi_1\\
0
\end{array} \right),\dots,\left(
\begin{array}{c}
0\\
\phi_{k+m}
\end{array} \right),\left(
\begin{array}{c}
\phi_{k+m}\\
0
\end{array} \right)\right\}\quad\textrm{if}\quad z\in\mathbb{R}.\label{real}
\end{align}
\end{lemma}
\begin{proof}
Let $z\in\rho(A)\cup\rho(A^*)$ and $x,y\in\cH$. If we set
\begin{displaymath}v = (A-z)^{-1}(A^*-z)^{-1}[(A+A^*-z)y - x]\quad\textrm{and}\quad u = zv+y,
\end{displaymath}
then a direct calculation shows that
\begin{displaymath}
(T-z)\left(
\begin{array}{c}
u\\
v
\end{array} \right) =
\left(
\begin{array}{c}
x\\
y
\end{array} \right),
\end{displaymath}
and therefore $\rho(T)\supset\rho(A)\cup\rho(A^*)$.

Suppose now that $z\in\sigma(A)\cup\sigma(A^*)$. Since $A$ is normal, there
exist normalised vectors $\psi_n$ such that either $(A-z)\psi_n\to 0$ or $(A^*-z)\psi_n\to 0$. It is then straightforward to show
that
\begin{displaymath}
(T-z)\left(
\begin{array}{c}
z\psi_n\\
\psi_n
\end{array} \right)\to \left(
\begin{array}{c}
0\\
0
\end{array} \right),
\end{displaymath}
and therefore $\sigma(T)\supset\sigma(A)\cup\sigma(A^*)$. The first assertion follows.

Now let $z\in\sigma_{\dis}(A)\cup\sigma_{\dis}(A^*)$. We assume that $z\notin\mathbb{R}$, the case where $z\in\mathbb{R}$ being treated
similarly. Let $\Gamma$ be a circle which does not pass through zero, and which encloses $z$ but no other member of
$\sigma(A)\cup\sigma(A^*)$. Using \eqref{mult0}, the spectral subspace
associated to $z$ is given by the range of the spectral projection
\begin{equation}\label{projection1}
Q(z) := -\frac{1}{2\pi i}\int_\Gamma\left(
\begin{array}{cc}
-w(A^*-wI)^{-1}(A-wI)^{-1}& (A^*-wI)^{-1}(A-wI)^{-1}A^*A\\
-(A^*-wI)^{-1}(A-wI)^{-1} & -w^{-1} + w^{-1}(A^*-wI)^{-1}(A-wI)^{-1}A^*A
\end{array} \right)~dw.
\end{equation}
Let $x,y,u,v\in\mathcal{H}$ with
\begin{equation}\label{orthvec}
\quad\left(
\begin{array}{c}
x\\
y
\end{array} \right)\perp\Span\left\{\left(
\begin{array}{c}
\phi_1\\
0
\end{array} \right),\left(
\begin{array}{c}
0\\
\phi_1
\end{array} \right),\dots,\left(
\begin{array}{c}
\phi_{k+m}\\
0
\end{array} \right),\left(
\begin{array}{c}
0\\
\phi_{k+m}
\end{array} \right)\right\},
\end{equation}
Using \eqref{projection1}, \eqref{orthvec} and the Cauchy-Goursat Theorem, we obtain
\begin{displaymath}
\left\langle Q(z)\left(
\begin{array}{c}
x\\
y
\end{array} \right),\left(
\begin{array}{c}
u\\
v
\end{array} \right)\right\rangle = -\frac{1}{2\pi i}\int_\Gamma\left\langle(T - w)^{-1}\left(
\begin{array}{c}
x\\
y
\end{array} \right),\left(
\begin{array}{c}
u\\
v
\end{array} \right)\right\rangle~dw = 0.
\end{displaymath}
We deduce that
\begin{displaymath}
\range(Q(z))\subseteq\Span\left\{\left(
\begin{array}{c}
\phi_1\\
0
\end{array} \right),\left(
\begin{array}{c}
0\\
\phi_1
\end{array} \right),\dots,\left(
\begin{array}{c}
\phi_{k+m}\\
0
\end{array} \right),\left(
\begin{array}{c}
0\\
\phi_{k+m}
\end{array} \right)\right\},
\end{displaymath}
and since
\begin{displaymath}
(T - z)\left(
\begin{array}{c}
z\phi_j\\
\phi_j
\end{array} \right) = \left(
\begin{array}{c}
0\\
0
\end{array} \right)\quad\textrm{and}\quad(T - \overline{z})\left(
\begin{array}{c}
\overline{z}\phi_j\\
\phi_j
\end{array} \right) = \left(
\begin{array}{c}
0\\
0
\end{array} \right)\quad\textrm{for}\quad j=1,\dots,k+m,
\end{displaymath}
the result follows.
\end{proof}

For an arbitrary finite dimensional subspace $\cL$ with corresponding orthogonal projection $P$, we consider the block operator matrix
\begin{displaymath}
S_{\cL} := \left(
\begin{array}{cc}
P(A+A^*) & -PA^*A\\
I & 0
\end{array} \right):\cL\oplus\cL\to\cL\oplus\cL.
\end{displaymath}

\begin{lemma}
Let $\cL$ be a finite dimensional subspace with corresponding orthogonal projection $P$, then $\sigma(S_{\cL}) = \Spec_2(A,\cL)$.
\end{lemma}
\begin{proof}
Let $z\in\sigma(S_{\cL})$, then there exist $\phi,\psi\in\cL$ such that
\begin{displaymath}
S_{\cL}\left(\begin{array}{c}
\psi\\
\phi
\end{array} \right)
=z\left(\begin{array}{c}
\psi\\
\phi
\end{array} \right)\quad\textrm{and}\quad\left(\begin{array}{c}
\psi\\
\phi
\end{array} \right)\ne\left(\begin{array}{c}
0\\
0
\end{array} \right).
\end{displaymath}
Therefore $\psi=z\phi$ and hence $zP(A+A^*)\phi - PA^*A\phi = z^2\phi$. It follows that
$\langle(A - z)\phi,(A-\overline{z})\psi\rangle = 0$ for all $\psi\in\cL$, and therefore $z\in\Spec_2(A,\cL)$.

Let $z\in\Spec_2(A,\cL)$, then there exists a $\phi\in\cL\backslash\{0\}$ such that $\langle(A - z)\phi,(A-\overline{z})\psi\rangle = 0$ for all $\psi\in\cL$.
It follows that $PA^*A\phi - P(A+A^*)\phi + z^2\phi = 0$ so that
\begin{displaymath}
S_{\cL}\left(\begin{array}{c}
z\phi\\
\phi
\end{array} \right)
=z\left(\begin{array}{c}
z\phi\\
\phi
\end{array} \right),
\end{displaymath}
and therefore $z\in\sigma(S_{\cL})$.
\end{proof}

It will be useful to note that for any non-zero $w\in\rho(S_{\cL})$ we have
\begin{equation}\label{proTn}
(S_{\cL}-w)^{-1}=\left(\begin{array}{cc}
-w[P(A^*-wI)(A-wI)]^{-1}& [P(A^*-wI)(A-wI)]^{-1}PA^*A\\
-[P(A^*-wI)(A-wI)]^{-1} & -w^{-1} + w^{-1}[P(A^*-wI)(A-wI)]^{-1}PA^*A
\end{array} \right).
\end{equation}
For a basis $\{\psi_1,\dots,\psi_d\}$ of $\cL$, we consider the matrices
\begin{equation}\label{matrices0}
B_{i,j} = \langle A\psi_j,A\psi_i\rangle,\quad L_{i,j} = \langle(A+A^*)\psi_j,\psi_i\rangle,\quad\textrm{and}\quad
M_{i,j} = \langle \psi_j,\psi_i\rangle.
\end{equation}
The matrices $B,L$ and $M$ each defines an operator on $\cL$ in a natural way:
\begin{equation}\label{matrices1}
B\psi = \sum_{i}\langle A\psi,A\psi_i\rangle\psi_i,\quad L\psi = \sum_{i}\langle(A+A^*)\psi,\psi_i\rangle\psi_i,\quad
\textrm{and}\quad M\psi = \sum_{i}\langle \psi,\psi_i\rangle\psi_i.
\end{equation}
We note that
\begin{equation}\label{matrices2}
S_{\cL} = \left(
\begin{array}{cc}
M^{-1} & 0\\
0 & M^{-1}
\end{array} \right)\left(
\begin{array}{cc}
L & -B\\
M & 0
\end{array} \right).
\end{equation}

\section{The limit set and Convergence Rates}

With the exception of the last assertion in Theorem \ref{th}, the results in Section 4.1 are known for self-adjoint
operators (see \cite{bo,bost,shar}). We
refine and extend these results to normal operators.

\subsection{The Limit Set}

\begin{lemma}\label{prebound}
Let $\cL$ be a finite dimensional subspace, then $\sigma(S_\cL)\subset\mathcal{Q}(\sigma(A))$ and for any $z\notin\mathcal{Q}(\sigma(A))$ we have
\begin{displaymath}
\vert\langle(A - z)\psi,(A-\overline{z})\psi\rangle\vert \ge \dist[z,\mathcal{Q}(\sigma(A))]^2\Vert\psi\Vert^2\quad\textrm{for all}\quad\psi\in\mathcal{H}.
\end{displaymath}
\end{lemma}
\begin{proof}
Suppose $z\notin\mathcal{Q}(\sigma(A))$. From Corollary \ref{cor2} we have $\dist(0,\conv(\sigma(A)_z))\ge\dist[z,\mathcal{Q}(\sigma(A))]^2$, then it follows that for some $\theta\in[0,2\pi)$ we have $\Re e^{i\theta}(\lambda - z)(\overline{\lambda} - z)\ge\dist[z,\mathcal{Q}(\sigma(A))]^2$ for all $\lambda\in \sigma(A)$. Thus
\begin{displaymath}
\Re e^{i\theta}\langle(A - z)\psi,(A-\overline{z})\psi\rangle = \int_{\sigma(A)}\Re e^{i\theta}(\lambda-z)(\overline{\lambda} - z)~d\langle E_\lambda\psi,\psi\rangle \ge \dist[z,\mathcal{Q}(\sigma(A))]^2\Vert\psi\Vert^2,
\end{displaymath}
from which both assertions follow.
\end{proof}

\begin{lemma}\label{ahat}
Let $\varepsilon>0$ and $\sigma(A)\backslash[\sigma_{\ess}(A)]_\varepsilon = \{z_1,\dots,z_m\}$. For any $e\in\sigma_{\ess}(A)$ the operator
\begin{equation}\label{perturbed}
\hat{A}:= A + \sum_{j=1}^m (e - z_j)E(\{z_j\})\quad\textrm{satisfies}\quad
\mathcal{Q}(\sigma(\hat{A}))\subseteq[\mathcal{Q}(\sigma_{\ess}(A))]_\varepsilon.
\end{equation}
\end{lemma}
\begin{proof}
Evidently, $\sigma(\hat{A})\subset[\sigma_{\ess}(A)]_\varepsilon$, therefore the assertion follows from Corollary \ref{cor2}.
\end{proof}

For a $z\in \sigma(S_{\cL})$ we denote the corresponding spectral subspace by $\mathcal{M}_{\cL}(\{z\})$.

\begin{lemma}\label{mult}
Let $\cL$ be a finite dimensional subspace, $\varepsilon>0$ and $\sigma(A)\backslash[\sigma_{\ess}(A)]_\varepsilon = \{z_1,\dots,z_m\}$. If $\cL(\{z_1,\dots,z_m\})\subseteq\cL$, then
\begin{equation}\label{a}
\sigma(S_{\cL})\cap\big(\mathbb{C}\backslash[\mathcal{Q}(\sigma_{\ess}(A))]_\varepsilon\big) = \{z_1,\overline{z}_1,\dots,z_m,\overline{z}_m\}.
\end{equation}
Moreover, $\mathcal{M}_{\cL}(z_j)$ and $\mathcal{M}_{\cL}(\overline{z}_j)$ are given by \eqref{imaginary} if $z_j\notin\mathbb{R}$, and by $\eqref{real}$ if $z_j\in\mathbb{R}$.
\end{lemma}
\begin{proof}
That $z_j,\overline{z}_j\in \sigma(S_{\cL})$ is obvious. Suppose $z\notin[\mathcal{Q}(\sigma_{\ess}(A))]_\varepsilon\cup\{z_1,\overline{z}_1,\dots,z_m,\overline{z}_m\}$ and  $z\in \sigma(S_{\cL})$.
Let  $\phi_1,\dots,\phi_s$ form a basis of eigenvectors for $\cL(\{z_1,\dots,z_m\})$. For some non-zero $\psi\in\cL$ we have  $\langle (A - z)\psi,(A - \overline{z})\phi\rangle = 0$ for all $\phi\in\mathcal{L}$. In particular, for any $1\le j\le s$ we have
\begin{displaymath}
0 = \langle (A - z)\psi,(A - \overline{z})\phi_j\rangle =
(z_k-z)(\overline{z}_k-z)\langle \psi,\phi_j\rangle\quad\textrm{for some}
\quad z_k\in\{z_1,\dots,z_m\},
\end{displaymath}
and we deduce that $\psi\perp\cL(\{z_1,\dots,z_m\})$. With $\hat{A}$ given by \eqref{perturbed}, it follows from Lemma \ref{prebound} and Lemma \ref{ahat} that $\langle(\hat{A} - z)\psi,(\hat{A}-\overline{z})\psi\rangle\ne 0$. However, since $\psi\perp\cL(\{z_1,\dots,z_m\})$ we have
$\langle (A - z)\psi,(A - \overline{z})\psi\rangle=\langle(\hat{A} - z)\psi,(\hat{A} -\overline{z})\psi\rangle$, and
\eqref{a} follows from the contradiction.

For the second assertion we assume that $z_j\notin\mathbb{R}$, the case
where $z_j\in\mathbb{R}$ being treated similarly. Let $z_j$ and $\overline{z}_j$ have multiplicities $k$ and $l$ (as eigenvalues of $A$),  respectively.
After a possible relabeling let $\cL(\{z_j\})=\Span\{\phi_1,\dots,\phi_k\}$ and $\cL(\{\overline{z}_j\})=\Span\{\phi_{k+1},\dots,\phi_{k+l}\}$. Evidently, we have
\begin{displaymath}
(S_{\cL} - z_j) \left(
\begin{array}{c}
z_j\phi_h\\
\phi_h
\end{array} \right)= \left(
\begin{array}{c}
0\\
0
\end{array} \right)\quad\textrm{for}\quad 1\le h\le k+l.
\end{displaymath}
Then if for some $x,y\in\cL$ and $1\le h\le k+l$ we have
\begin{displaymath}(S_{\cL} - z_j) \left(
\begin{array}{c}
x\\
y
\end{array} \right)= \left(
\begin{array}{c}
z_j\phi_h\\
\phi_h
\end{array} \right),\quad\textrm{then}\quad x = \phi_h+z_jy\quad\textrm{and}\quad
P(A-z_j)(A^*-z_j)y=2i\Im z_j\phi_h
\end{displaymath}
where $P$ is the orthogonal projection onto $\cL$. The last term implies that $y=0$, therefore $x=\phi_h$ and $P(A+A^*-z_j)\phi_h = z_j\phi_h$ which is a contradiction. We have shown that
\begin{displaymath}
\mathcal{M}_{\cL}(\{z_j\})\supseteq\Span\left\{\left(
\begin{array}{c}
z_j\phi_1\\
\phi_1
\end{array} \right),\dots,\left(
\begin{array}{c}
z_j\phi_{k+l}\\
\phi_{k+l}
\end{array} \right)\right\}
\end{displaymath}
and that equality can only fail if
\begin{equation}\label{laste}
(S_{\cL} - z_j) \left(
\begin{array}{c}
z_j\phi\\
\phi
\end{array} \right)= \left(
\begin{array}{c}
0\\
0
\end{array} \right)\quad\textrm{for some}\quad\phi\in\cL,\quad\phi\notin\cL(\{z_1,\overline{z}_1,\dots,z_m,\overline{z}_m\}).
\end{equation}
Suppose \eqref{laste} holds and let $\psi=(I-E(\{z_j,\overline{z}_j\}))\phi$. Then $\psi\in\cL\backslash\{0\}$ and
\begin{displaymath}
(S_{\cL} - z_j) \left(
\begin{array}{c}
z_j\psi\\
\psi
\end{array} \right)= \left(
\begin{array}{c}
0\\
0
\end{array} \right)\quad\Rightarrow\quad P(A-z_j)(A^*-z_j)\psi = 0.
\end{displaymath}
Clearly  $\psi\perp\cL(\{z_j,\overline{z}_j\})$, and arguing as above it follows that $\psi\perp\cL(\{z_1,\dots,z_m\})$. Therefore $\langle (A - z)\psi,(A - \overline{z})\psi\rangle=\langle(\hat{A} - z)\psi,(\hat{A} -\overline{z})\psi\rangle$
and again Lemma \ref{prebound} and Lemma \ref{ahat} yield a contradiction.
\end{proof}

For an $\varepsilon>0$ with $\sigma(A)\backslash[\sigma_{\ess}(A)]_\varepsilon = \{z_1,\dots,z_m\}$ we set
$M_0 = \max\{\vert e-z_j\vert:e\in\sigma_{\ess}(A):1\le j\le m\}$ and define the following functions acting on $\mathbb{C}\backslash\big(\sigma(A)\cup \sigma(A^*)\cup[\mathcal{Q}(\sigma_{\ess}(A))]_\varepsilon\big)$
\begin{align*}
f_1(z) &= \max\{\vert z- z_j\vert,\vert z - \overline{z}_j\vert:1\le j\le m\},\\
f_2(z)&=\min\{\vert z - z_j\vert\vert z - \overline{z}_j\vert: 1\le j\le m\},\\
f_3(z)&=\max\{\vert z-e\vert\vert z-\overline{e}\vert:e\in\sigma_{\ess}(A)\},\\
f_4(z)&=\min\{\vert z-e\vert\vert z-\overline{e}\vert:e\in\sigma_{\ess}(A)\},
\end{align*}
and
\begin{equation}\label{albe}
\alpha(z) = \min\left\{1,\frac{f_2(z)}{f_3(z)}\right\}\quad\textrm{and}
\quad\beta(z) = \frac{2M_0f_1(z) + M_0^2}{f_4(z)}\left(1+\frac{\Vert (A-z)(A^*-z)\Vert}{\dist[z,[\mathcal{Q}(\sigma_{\ess}(A))]_\varepsilon]^2}\right).
\end{equation}

\begin{theorem}\label{bound}
Let $\cL$ be a finite dimensional subspace with corresponding orthogonal projection $P$. Let $\varepsilon>0$ and $\sigma(A)\backslash[\sigma_{\ess}(A)]_\varepsilon = \{z_1,\dots,z_m\}$.
For any $z\notin \sigma(A)\cup \sigma(A^*)\cup[\mathcal{Q}(\sigma_{\ess}(A))]_\varepsilon$ we have
\begin{equation}\label{sp}
\Vert P(A - z)(A^*-z)P\psi\Vert\ge
\dist[z,[\mathcal{Q}(\sigma_{\ess}(A))]_\varepsilon]^{2}\Big(\alpha(z) - \beta(z)\delta(\cL(\{z_1,\dots,z_m\}),\cL)\Big)\Vert P\psi\Vert
\end{equation}
for all $\psi\in\mathcal{H}$.
\end{theorem}
\begin{proof}
Let $\hat{A}$ be the operator defined in Lemma \ref{ahat} for some arbitrary $e\in\sigma_{\ess}(A)$.
For convenience we will write $A(z):= (A-z)(A^*-z)$ and $\hat{A}(z):= (\hat{A}-z)(\hat{A}^*-z)$.
Consider the following finite rank operator
\begin{displaymath}
K(z):=\hat{A}(z) - A(z) = \sum_{j=1}^m\big((\overline{z}_j-z)(e-z_j)+ (z_j-z)(\overline{e}-\overline{z}_j)+\vert e - z_j\vert^2\big)E(\{z_j\}).
\end{displaymath}
Evidently, $\Vert K(z)\Vert = \max\{\Vert K(z)\phi\Vert:\phi\in\cL(\{z_1,\dots,z_m\})\textrm{ and }\Vert\phi\Vert=1\}$, and for any $\phi\in\cL(\{z_1,\dots,z_m\}$ we have
\begin{equation}\label{knorm}
\Vert K(z)\phi\Vert^2 \le (2M_0f_1(z) + M_0^2)^2\Vert\phi\Vert^2\quad\Rightarrow\quad\Vert K(z)\Vert \le 2M_0f_1(z) + M_0^2.
\end{equation}
Let $E = E(\{z_1,\dots,z_m\})$, then for any $\psi\in\mathcal{H}$
\begin{align*}
\Vert\psi - \hat{A}(z)^{-1}K(z)\psi\Vert^2 &= \Vert(I - E)\psi\Vert^2 + \Vert E\psi -\hat{A}(z)^{-1}[\hat{A}(z) - A(z)]E\psi\Vert^2\\
&=\Vert(I - E)\psi\Vert^2 + \Vert\hat{A}(z)^{-1}A(z)E\psi\Vert^2\\
&\ge\Vert(I - E)\psi\Vert^2 + \frac{f_2(z)^2}{f_3(z)^2}\Vert E\psi\Vert^2\ge\alpha(z)^{2}\Vert \psi\Vert^2.
\end{align*}
Using Lemma \ref{prebound} and Lemma \ref{ahat} we have
\begin{displaymath}\Vert[P\hat{A}(z)|_{\cL}]^{-1}\Vert\le\dist[z,\mathcal{Q}(\sigma(\hat{A}))]^{-2}\le
\dist[z,[\mathcal{Q}(\sigma_{\ess}(A))]_\varepsilon]^{-2}.
\end{displaymath}
Note also that since $A$ is normal and $\sigma(\hat{A})\subset\sigma(A)$, it follows that $\Vert\hat{A}(z)\Vert\le\Vert A(z)\Vert$.
Combining these two estimates with \eqref{knorm} we obtain for any $\psi\in\mathcal{H}$
\begin{align*}
\Vert\hat{A}(z)^{-1}K(z)\psi - [P\hat{A}(z)|_{\cL}]^{-1}PK(z)\psi\Vert
&\le \Vert(I- P)\hat{A}(z)^{-1}K(z)\psi\Vert\\
&+ \Vert P\hat{A}(z)^{-1}K(z)\psi - [P\hat{A}(z)|_{\cL}]^{-1}PK(z)\psi\Vert\\
&\le \frac{\Vert(I- P)K(z)E\psi\Vert}{\vert e-z\vert\vert\overline{e}-z\vert}\\
&+ \Vert[P\hat{A}(z)|_{\cL}]^{-1}\Vert\Vert\hat{A}(z)P\hat{A}(z)^{-1}K(z)E\psi- K(z)E\psi\Vert\\
&\le \frac{\Vert(I- P)K(z)\Vert}{f_4(z)}\Vert\psi\Vert\\
&+ \Vert[P\hat{A}(z)|_{\cL}]^{-1}\Vert\Vert\hat{A}(z)\Vert\Vert K(z)\Vert\Vert (I-P)\hat{A}(z)^{-1}E\psi\Vert\\
&\le \beta(z)\delta(\cL(\{z_1,\dots,z_m\}),\cL)\Vert\psi\Vert.
\end{align*}
Finally, we have for any $\psi\in\mathcal{H}$
\begin{align*}
\Vert PA(z)P\psi\Vert &= \Vert P\hat{A}(z)P\psi - PK\psi\Vert\ge\Vert[P\hat{A}(z)|_{\cL}]^{-1}\Vert^{-1}\Vert P\psi - [P\hat{A}(z)|_{\cL}]^{-1}PKP\psi\Vert\\
&\ge \dist[z,\mathcal{Q}(\sigma(\hat{A}))]^{2}\big(\Vert P\psi - \hat{A}(z)^{-1}KP\psi\Vert - \Vert \hat{A}(z)^{-1}KP\psi - [P\hat{A}(z)|_{\cL}]^{-1}PKP\psi\Vert\big)\\
&\ge\dist[z,[\mathcal{Q}(\sigma_{\ess}(A))]_\varepsilon]^{2}\Big(\alpha(z) - \beta(z)\delta(\cL(\{z_1,\dots,z_m\}),\cL)\Big)\Vert P\psi\Vert.
\end{align*}
\end{proof}


For a sequence of subspaces $(\cL_n)\in\Lambda$ we shall write $S_n$ instead of $S_{\cL_n}$.
For a $z\in \sigma(S_n)$ we denote the corresponding spectral subspace by $\mathcal{M}_{n}(\{z\})$ instead of $\mathcal{M}_{\cL_n}(\{z\})$.
For each $n\in\mathbb{N}$ the orthogonal projection onto $\cL_n$ will be denoted $P_n$.

\begin{corollary}\label{lim2b}
Let $(\cL_n)\in\Lambda$, then
\begin{displaymath}
\Big(\lim_{n\to\infty}\sigma(S_n)\Big)\backslash\mathcal{Q}(\sigma_{\ess}(A))\subset \sigma_{\dis}(A)\cup \sigma_{\dis}(A^*).
\end{displaymath}
\end{corollary}
\begin{proof}
Let $\varepsilon>0$, $\sigma(A)\backslash[\sigma_{\ess}(A)]_\varepsilon = \{z_1,\dots,z_m\}$, and
$\mathcal{N}\cap\big(\sigma(A)\cup \sigma(A^*)\cup[\mathcal{Q}(\sigma_{\ess}(A))]_\varepsilon\big) = \varnothing
$ where $\mathcal{N}$ is a compact set. We set $\alpha = \min\{\alpha(z):z\in\mathcal{N}\}$ and $\beta = \max\{\beta(z):z\in\mathcal{N}\}$.
There exists an $N\in\mathbb{N}$ such that $\dist[\cL(\{z_1,\dots,z_m\}),\cL_n] < \alpha/\beta$ for all $n\ge N$. Then it
follows from Theorem \ref{bound} that $\mathcal{N}\cap\sigma(S_n)=\varnothing$ for all $n\ge N$.
\end{proof}

For $z\in \sigma(A)\backslash\mathcal{Q}(\sigma_{\ess}(A))$ with
$\dist[z,\big(\mathcal{Q}(\sigma_{\ess}(A))\cup \sigma(A)\cup \sigma(A^*)\big)\backslash\{z\}] = \delta$,
we denote by $\mathcal{M}_n(\{z\},r)$ the spectral subspace of $S_n$ associated to those eigenvalues enclosed by the circle
$\Gamma$ with center $z$ and radius $r>0$. We will always assume that $r<\delta$, and that $\Gamma\cap\mathbb{R}=\varnothing$ if $z\notin\mathbb{R}$ and that $\Gamma$ does not pass through zero if $z\in\mathbb{R}$. The corresponding spectral projection we denote by $Q_n(\{z\},r)$. It will be useful to extend $Q_n(\{z\},r)$ in the following way
\begin{displaymath}
\hat{Q}_n(\{z\},r):=Q_n(\{z\},r) \left(\begin{array}{cc}
P_n & 0\\
0 & P_n
\end{array} \right):\mathcal{H}\oplus\cH\to\cL_n\oplus\cL_n,
\end{displaymath}
therefore $\range(\hat{Q}_n(\{z\},r)) = \range(Q_n(\{z\},r)) = \mathcal{M}_n(\{z\},r)$. By Lemma \ref{eigspaces} we have $z\in\sigma(T)$ with corresponding
spectral subspace given by $\mathcal{M}(\{z\})$ (see \eqref{imaginary} and \eqref{real}) which is the range of the spectral projection $Q(z)$ (see \eqref{projection1}).

\begin{theorem}\label{th}
Let $(\cL_n)\in\Lambda$, $z\in \sigma(A)\backslash\mathcal{Q}(\sigma_{\ess}(A))$ and fix $r,\Gamma$ as above. For all
sufficiently large $n\in\mathbb{N}$, we have $\Gamma\subset\rho(S_n)$ and
$\dim \mathcal{M}_n(\{z\},r) = \dim \mathcal{M}(\{z\})$. Moreover, we have
$\hat{Q}_n(\{z\},r)\stackrel{s}{\longrightarrow}Q(z)$ as $n\to\infty$.
\end{theorem}
\begin{proof}
Choose $\varepsilon>0$ sufficiently small so that $\Gamma\cap[\mathcal{Q}(\sigma_{\ess}(A))]_\varepsilon=\varnothing$.
Let $\sigma(A)\backslash[\sigma_{\ess}(A)]_\varepsilon = \{z_1,\dots,z_m\}$. Note that
$z\in\{z_1,\dots,z_m\}$ follows from Corollary \ref{cor2}. Let $\phi_1,\dots,\phi_s$ be an orthonormal basis for
$\cL(\{z_1,\dots,z_m\})$. Set $\psi_{n,j}=P_n\phi_j$ and $\psi_{n,j}(t) = t\psi_{n,j} + (1-t)\phi_j$ where $t\in[0,1]$.
There exists an $N_0\in\mathbb{N}$, such that whenever $n>N_0$ there are vectors
$\{\psi_{n,s+1},\dots,\psi_{n,\tilde{n}}\}\in\mathcal{L}_n$ (where $\tilde{n}=\dim(\cL_n)$) for which
$\{\psi_{n,1}(t),\dots,\psi_{n,s}(t),\psi_{n,s+1},\psi_{n,\tilde{n}}\}$ is a linearly independent set for all
$t\in[0,1]$; see \cite[Lemma 3.3]{bost}. Let $P_n(t)$ be the orthogonal projection onto
$\cL_n(t):=\Span\{\psi_{n,1}(t),\dots,\psi_{n,s}(t),\psi_{n,s+1},\psi_{n,\tilde{n}}\}$, and consider the following family of
block operator matrices
\begin{displaymath}
S_n(t) := \left(
\begin{array}{cc}
P_n(t)(A+A^*) & -P_n(t)A^*A\\
I & 0
\end{array} \right):\cL_n(t)\oplus\cL_n(t)\to\cL_n(t)\oplus\cL_n(t).
\end{displaymath}
With $\alpha(\cdot)$ and $\beta(\cdot)$ given by \eqref{albe}, set $\alpha = \min\{\alpha(w):w\in\Gamma\}$ and $\beta = \max\{\beta(w):w\in\Gamma\}$.
It follows from the fact that $\cL(\{z_1,\dots,z_m\})$ is finite dimensional and $(\cL_n)\in\Lambda$, that there exists an $N_1\in\mathbb{N}$ such that
$\delta(\cL(\{z_1,\dots,z_m\}),\cL_n) =: \delta_n  < \alpha/\beta$ for all $n\ge N_1$. It is easily verified that
$\delta(\cL(\{z_1,\dots,z_m\}),\cL_n(t))\le\delta(\cL(\{z_1,\dots,z_m\}),\cL_n)$ for any $t\in[0,1]$. It now follows from
Theorem \ref{bound} that $\Gamma\subset\rho(S_n(t))$ for all $t\in[0,1]$ and $n\ge N_0,N_1$.
The first assertion follows.

Evidently, the spectral projection associated to $S_n(t)$ and those elements from $\sigma(S_n(t))$ enclosed by $\Gamma$
depends continuously on $t\in[0,1]$. The second assertion now follows from Lemma \ref{mult} and \cite[Lemma 1.4.10]{katopert}.

For the last assertion let $\psi\in\mathcal{H}$. Then using \eqref{mult0} and \eqref{proTn} we obtain
\begin{align*}
\left\Vert\left[\hat{Q}_n(\{z\},r) - Q(z)\right]\left(
\begin{array}{c}
\psi\\
0
\end{array} \right)\right\Vert
&\le \frac{1}{2\pi}\int_{\Gamma}\left\Vert\left(
\begin{array}{cc}
-w[P_nA(w)]^{-1}P_n\psi+wA(w)^{-1}\psi\\
-[P_nA(w)]^{-1}P_n\psi+A(w)^{-1}\psi
\end{array} \right)\right\Vert~dw\\
&\le \frac{1}{2\pi}\int_{\Gamma}(1+\vert w\vert)\Vert
[P_nA(w)]^{-1}P_n\psi-A(w)^{-1}\psi\Vert~dw.
\end{align*}
Set $d=\min\{\dist[z,[\mathcal{Q}(\sigma_{\ess}(A))]_\varepsilon]^{2}:z\in\Gamma\}$. Since
$\delta_n\to 0$, we have $1/d(\alpha - \beta\delta_n)\le 2/d\alpha =:c$ for all sufficiently large $n\in\mathbb{N}$.
Combining this estimate with Theorem \ref{bound} yields
\begin{displaymath}
\Vert[P_n(A - z)(A^*-z)|_{\mathcal{L}_n}]^{-1}\Vert\le c
\quad\textrm{for all~} z\in\Gamma\textrm{~and sufficiently large~}n\in\mathbb{N}.
\end{displaymath}
Now consider the following sequence of functions
\begin{align*}
g_n(w) &:= \Vert[P_nA(w)]^{-1}P_n\psi-A(w)^{-1}\psi\Vert\\
&\le\Vert[P_nA(w)]^{-1}P_n\psi-P_nA(w)^{-1}\psi\Vert+\Vert(I-P_n)A(w)^{-1}\psi\Vert\\
&\le c\Vert P_n\psi-P_nA(w)P_nA(w)^{-1}\psi\Vert+\Vert(I-P_n)A(w)^{-1}\psi\Vert
\end{align*}
with $\Dom(g_n) = \Gamma$. It is clear that the functions $g_n$ converge pointwise to zero. For any fixed
$w\in\Gamma$ and sequence $(w_n)\in\Gamma$ with $w_n\rightarrow w$, we have
\begin{equation}\label{gcon}
g_n(w_n)\le c\Vert P_n\psi-P_nA(w_n)P_nA(w_n)^{-1}\psi\Vert+\Vert(I-P_n)A(w_n)^{-1}\psi\Vert.
\end{equation}
Clearly, the right hand side of \eqref{gcon} converges to zero, from which it follows that
the functions $g_n$ converge uniformly to zero (see \cite[Theorem 7.3.5]{stri}). Therefore
\begin{align*}
\left\Vert\left[\hat{Q}_n(z) - Q(z)\right]\left(
\begin{array}{c}
\psi\\
0
\end{array} \right)\right\Vert
\to 0,\quad\textrm{and similarly}\quad\left\Vert\left[\hat{Q}_n(z) - Q(z)\right]\left(
\begin{array}{c}
0\\
\psi
\end{array} \right)\right\Vert
\to 0.
\end{align*}
\end{proof}

\begin{corollary}
Let $(\cL_n)\in\Lambda$, then $(\lim \sigma(S_n))\backslash\mathcal{Q}(\sigma_{\ess}(A))=\big(\sigma_{\dis}(A)\cup \sigma_{\dis}(A^*)\big)\backslash\mathcal{Q}(\sigma_{\ess}(A))$. If $A$ is self-adjoint then
$(\lim \sigma(S_n))\backslash\mathcal{Q}(\sigma_{\ess}(A))=\sigma_{\dis}(A)$.
\end{corollary}
\begin{proof}
The first assertion is an immediate consequence of Corollary \ref{lim2b} and the second assertion of Theorem \ref{th}. The second assertion follows from  Corollary \ref{lim2b}, the second assertion of Theorem \ref{th},
and the fact that $\sigma(A)\subset\mathbb{R}$ so that $\sigma_{\dis}(A)\cap\mathcal{Q}(\sigma_{\ess}(A))=\emptyset$.
\end{proof}

\subsection{Convergence Rates}

\begin{example}
Let $(\phi_n)_{n\in\mathbb{N}}$ form an orthonormal basis for $\mathcal{H}$, $P\psi = \langle\psi,\phi_1\rangle\phi_1$ and $A = I - P$. Then $A$ is a bounded self-adjoint
operator with $\sigma_{\ess}(A) = \{1\}$ and $\sigma_{\dis}(A) = \{0\}$. Let $\cL_n = \Span\{\phi_2,\dots\phi_{n-1},\psi_n\}$ where
$\psi_n = \alpha_n\phi_1 + \varepsilon_n\phi_n$, $\alpha_n,\varepsilon_n\in\mathbb{R}$, $\alpha_n^2 + \varepsilon_n^2 = 1$ and $\varepsilon_n\to 0$. Then $\delta(\cL(\{0\}),\cL_n) = \dist
[\phi_1,\cL_n] = \varepsilon_n$, $\sigma(S_n) = \{\varepsilon_n^2\pm i(\varepsilon_n^2 - \varepsilon_n^4)^{\frac{1}{2}},1\}$, and therefore $\dist[0,\sigma(S_n)] = \varepsilon_n$.
For the spectral subspaces we have
\begin{displaymath}
\mathcal{M}(\{0\}) = \Span\left\{\left(
\begin{array}{c}
0\\
\phi_1
\end{array} \right),\left(
\begin{array}{c}
\phi_1\\
0
\end{array} \right)\right\}
\end{displaymath}
and for any $\varepsilon_n<r<1$
\begin{align*}
\mathcal{M}_n(\{0\},r)&=
\Span\left\{\left(
\begin{array}{c}
\big(\varepsilon_n^2\pm i(\varepsilon_n^2 - \varepsilon_n^4)^{\frac{1}{2}}\big)\psi_n\\
\psi_n
\end{array} \right),\left(
\begin{array}{c}
\big(\varepsilon_n^2\pm i(\varepsilon_n^2 - \varepsilon_n^4)^{\frac{1}{2}}\big)\psi_n\\
\psi_n
\end{array} \right)\right\}\\
&=
\Span\left\{\left(
\begin{array}{c}
\psi_n\\
0
\end{array} \right),\left(
\begin{array}{c}
0\\
\psi_n
\end{array} \right)\right\}.
\end{align*}
from which we easily obtain $\hat\delta(\mathcal{M}_n(\{0\},r),\mathcal{M}(\{0\})) =\varepsilon_n$.
\end{example}

The next theorem shows that $\hat\delta(\mathcal{M}_n(\{z\},r),\mathcal{M}(\{z\})) = \mathcal{O}(\delta(\cL(\{z,\overline{z}\}),\cL_n))$ and
$\dist[z, \sigma(S_n)] = \mathcal{O}(\delta(\cL(\{z,\overline{z}\}),\cL_n))$. The example above shows that these convergence rates are sharp. We also note that this
eigenvalue convergence rate has previously been observed in computations for a bounded self-adjoint operator (see \cite[Section 3.2]{bo2}).

For a $z\in \sigma(A)\backslash\mathcal{Q}(\sigma_{\ess}(A))$ and $r,\Gamma$ as above, let $\varepsilon>0$ be as in the proof of Theorem \ref{th}. We set $M_1 = \max\{\Vert A(w)\Vert:w\in\Gamma\}$,
$M_2 = \max\{\Vert A(w)^{-1}A^*A\Vert:w\in\Gamma\}$, $c$ as in the proof of Theorem \ref{th}, and
\begin{displaymath}
M_3 = \frac{(\vert z\vert^2 + \vert z\vert + r\vert z\vert)(cM_1 + 1)}{r} + \frac{(1 + \vert z\vert - r)(cM_1M_2r^2 + cM_1\vert z\vert^2+\vert z\vert^2)}{(\vert z\vert - r)r}
+ \frac{r}{\vert z\vert - r}.
\end{displaymath}

\begin{theorem}\label{lim3}
Let $(\cL_n)\in\Lambda$, $z\in \sigma(A)\backslash\mathcal{Q}(\sigma_{\ess}(A))$, $\delta(\cL(\{z,\overline{z}\}),\cL_n) =\varepsilon_n$ and fix $r,\Gamma$ as above, then for all sufficiently large $n\in\mathbb{N}$
\begin{displaymath}
\hat\delta(\mathcal{M}_n(\{z\},r),\mathcal{M}(\{z\})) \le \frac{M_3\varepsilon_n}{1 - M_3\varepsilon_n}\quad\textrm{and}
\quad\dist[z, \sigma(S_n)] \le (1+\vert z\vert)(1+\vert z\vert + r)^{\frac{1}{2}}\frac{M_3\varepsilon_n}{1 - M_3\varepsilon_n}.
\end{displaymath}
\end{theorem}
\begin{proof}
We assume that $z\notin\mathbb{R}$, the case where $z\in\mathbb{R}$ being treated similarly. From Lemma \ref{eigspaces} we have
\begin{displaymath}
\mathcal{M}(\{z\}) = \Span\left\{\left(
\begin{array}{c}
z\phi_1\\
\phi_1
\end{array} \right),\dots,\left(
\begin{array}{c}
z\phi_{k+m}\\
\phi_{k+m}
\end{array} \right)\right\}
\end{displaymath}
where $\cL(\{z\})=\Span\{\phi_1,\dots\phi_k\}$ and $\cL(\{\overline{z}\})=\Span\{\phi_{k+1},\dots\phi_{k+m}\}$. Let
\begin{displaymath}
\left(
\begin{array}{c}
z\phi\\
\phi
\end{array} \right)\in\mathcal{M}(\{z\})\quad\textrm{with}\quad\left\Vert\left(
\begin{array}{c}
z\phi\\
\phi
\end{array} \right)\right\Vert = 1.
\end{displaymath}
Let $\varepsilon,c>0$ be as in the proof of Theorem \ref{th}, then for all sufficiently large $n\in\mathbb{N}$ we have
$\Gamma\subset\rho(S_n)$ and $\Vert[P_nA(z)|_{\mathcal{L}_n}]^{-1}\Vert\le c$ for all  $z\in\Gamma$.
Using \eqref{mult0} and \eqref{proTn} we obtain
\begin{align*}
\left\Vert\left[\hat{Q}_n(\{z\},r) - Q(z)\right]\left(
\begin{array}{c}
z\phi\\
0
\end{array} \right)\right\Vert
&\le \frac{\vert z\vert}{2\pi}\int_{\Gamma}(1+\vert w\vert)\Vert
[P_nA(w)|_{\cL_n}]^{-1}P_n\phi-A(w)^{-1}\phi\Vert~dw\\
&\le \big(r\vert z\vert^2 + r\vert z\vert + r^2\vert z\vert\big)\max_{w\in\Gamma}\Vert
[P_nA(w)|_{\cL_n}]^{-1}P_n\phi-A(w)^{-1}\phi\Vert,
\end{align*}
where
\begin{align*}
\Vert[P_nA(w)|_{\cL_n}]^{-1}P_n\phi-A(w)^{-1}\phi\Vert &\le c\Vert P_n\phi-P_nA(w)P_nA(w)^{-1}\phi\Vert+\Vert(I-P_n)A(w)^{-1}\phi\Vert\\
&\le c\Vert A(w)\phi - A(w)P_n\phi\Vert/r^2+\Vert(I-P_n)\phi\Vert/r^2\\
&\le \delta(\cL(\{z,\overline{z}\}),\cL_n)\Vert\phi\Vert(cM_1 + 1)/r^2
\end{align*}
Similarly we have
\begin{align*}
\left\Vert\left[\hat{Q}_n(\{z\},r) - Q(z)\right]\left(
\begin{array}{c}
0\\
\phi
\end{array} \right)\right\Vert
&\le \frac{1}{2\pi}\int_\Gamma(1+\vert w\vert^{-1})\Vert [P_nA(w)|_{\cL_n}]^{-1}P_nA^*AP_n\phi - A(w)^{-1}A^*A\phi\Vert~dw\\
&+ \frac{1}{2\pi}\int_\Gamma\vert w\vert^{-1}\Vert P_n\phi -\phi\Vert~dw\\
&\le\frac{r(1+\vert z\vert  - r)}{\vert z\vert - r}\max_{w\in\Gamma}\Vert [P_nA(w)|_{\cL_n}]^{-1}P_nA^*AP_n\phi - A(w)^{-1}A^*A\phi\Vert\\
&+ \frac{r\delta(\cL(\{z,\overline{z}\}),\cL_n)}{\vert z\vert - r}\Vert\phi\Vert,
\end{align*}
where
\begin{align*}
\Vert [P_nA(w)|_{\cL_n}]^{-1}P_nA^*AP_n\phi - A(w)^{-1}A^*A\phi\Vert &\le c\Vert A^*AP_n\phi - A(w)P_nA(w)^{-1}A^*A\phi\Vert\\
&+ \vert z\vert^2\Vert(I-P_n)\phi\Vert/r^2\\
&\le cM_1\Vert A(w)^{-1}A^*AP_n\phi - P_nA(w)^{-1}A^*A\phi\Vert\\
&+ \vert z\vert^2\Vert(I-P_n)\phi\Vert/r^2\\
&\le cM_1\Vert A(w)^{-1}A^*A(I-P_n)\phi\Vert\\
&+cM_1\Vert (I-P_n)A(w)^{-1}A^*A\phi\Vert+ \vert z\vert^2\Vert(I-P_n)\phi\Vert/r^2\\
&\le\delta(\cL(\{z,\overline{z}\}),\cL_n)\Vert\phi\Vert (cM_1M_2r^2 + cM_1\vert z\vert^2 + \vert z\vert^2)/r^2.
\end{align*}
Combining these estimates we have
\begin{align*}
\left\Vert\hat{Q}_n(z)\left(
\begin{array}{c}
z\phi\\
\phi
\end{array} \right)-
\left(
\begin{array}{c}
z\phi\\
\phi
\end{array} \right)
\right\Vert &= \left\Vert\left[\hat{Q}_n(z) - Q(z)\right]\left(
\begin{array}{c}
z\phi\\
\phi
\end{array} \right)\right\Vert = M_3\varepsilon_n,
\end{align*}
and therefore $\delta(\mathcal{M}(\{z\}),\mathcal{M}_n(\{z\},r))\le M_3\varepsilon_n$. From Theorem \ref{th} we have $\dim\mathcal{M}(\{z\}) = \dim\mathcal{M}_n(\{z\},r)$ which combined with \cite[Lemma 213]{kato} yields the estimate
\begin{displaymath}
\delta(\mathcal{M}_n(\{z\},r),\mathcal{M}(\{z\}))\le \frac{\delta(\mathcal{M}(\{z\}),\mathcal{M}_n(\{z\},r))}{1-\delta(\mathcal{M}(\{z\}),\mathcal{M}_n(\{z\},r))}.
\end{displaymath}
The first assertion follows.

For the second assertion we let
\begin{displaymath}
S_n\left(
\begin{array}{c}
\hat{z}\psi\\
\psi
\end{array} \right) =
\hat{z}\left(
\begin{array}{c}
\hat{z}\psi\\
\psi
\end{array} \right)\quad\textrm{for some}\quad\left(
\begin{array}{c}
\hat{z}\psi\\
\psi
\end{array} \right)\in\mathcal{M}_n(\{z\},r)\quad\textrm{with}\quad\left\Vert\left(
\begin{array}{c}
\hat{z}\psi\\
\psi
\end{array} \right)\right\Vert = 1.
\end{displaymath}
Then $\vert\hat{z}\vert\le\vert z\vert + r$ and for some $\phi\in\mathcal{L}(\{z,\overline{z}\})$ we have
\begin{displaymath}
\left(
\begin{array}{c}
z\phi\\
\phi
\end{array} \right) \in\mathcal{M}(\{z\})\quad\textrm{and}\quad\left\Vert\left(
\begin{array}{c}
z\phi\\
\phi
\end{array} \right)-\left(
\begin{array}{c}
\hat{z}\psi\\
\psi
\end{array} \right)\right\Vert\le\frac{M_3\varepsilon_n}{1 - M_3\varepsilon_n},
\end{displaymath}
from which the second assertion follows.\end{proof}

\begin{corollary}\label{susp}
Let $(\cL_n)\in\Lambda$, $z\in \sigma(A)\backslash\mathcal{Q}(\sigma_{\ess}(A))$, $\delta(\cL(\{z,\overline{z}\}),\cL_n) =\varepsilon_n$ and fix $r,\Gamma$ as above. Let
\begin{displaymath}
P_{+}\left(
\begin{array}{c}
u\\
v
\end{array}\right) = u,\quad P_{-}\left(
\begin{array}{c}
u\\
v
\end{array}\right) = v,\quad\textrm{and}\quad\mathcal{M}^\pm_n(\{z\},r) = \{P_\pm u:u\in\mathcal{M}_n(\{z\},r)\},
\end{displaymath}
then $\hat\delta(\mathcal{M}_n^\pm(\{z\},r),\cL(\{z,\overline{z}\})) = \mathcal{O}(\varepsilon_n)$.
\end{corollary}
\begin{proof}
We assume that $z\notin\mathbb{R}$, the case where $z\in\mathbb{R}$ being treated similarly. It follows from Theorem \ref{lim3} that for all sufficiently large $n\in\mathbb{N}$ and any $\phi\in\cL(\{z,\overline{z}\})$ with
$\Vert\phi\Vert = 1$, we have
\begin{displaymath}
\left(
\begin{array}{c}
z\phi\\
\phi
\end{array}\right)\in\mathcal{M}(\{z\})\quad\textrm{and}\quad\left\Vert\left(
\begin{array}{c}
z\phi\\
\phi
\end{array}\right)-\left(
\begin{array}{c}
u\\
v
\end{array}\right)\right\Vert\le\frac{\sqrt{1+\vert z\vert^2}M_3\varepsilon_n}{1 - M_3\varepsilon_n}\quad\textrm{for some}\quad\left(
\begin{array}{c}
u\\
v
\end{array}\right)\in\mathcal{M}_n(\{z\},r).
\end{displaymath}
We deduce that
\begin{equation}\label{perkato}
\delta(\cL(\{z,\overline{z}\}),\mathcal{M}_n^+(\{z\},r)) \le \frac{\sqrt{1+\vert z\vert^2}M_3\varepsilon_n}{\vert z\vert(1 - M_3\varepsilon_n)}\quad\textrm{and}\quad\delta(\cL(\{z,\overline{z}\}),\mathcal{M}_n^-(\{z\},r)) \le \frac{\sqrt{1+\vert z\vert^2}M_3\varepsilon_n}{1 - M_3\varepsilon_n}.\end{equation}
The result now follows from \eqref{perkato} and the estimate
\begin{displaymath}
\delta(\mathcal{M}^\pm_n(\{z\},r),\cL(\{z,\overline{z}\})\le \frac{\delta(\cL(\{z,\overline{z}\},\mathcal{M}^\pm_n(\{z\},r))}{1-\delta(\cL(\{z,\overline{z}\},\mathcal{M}^\pm_n(\{z\},r))}
\end{displaymath}
(see \cite[Lemma 213]{kato}).
\end{proof}

\section{Eigenvalue Enclosures}

Recall the finite-section method which we discussed briefly in the introduction. The problem with this method is that we can
encounter sequences $z_n\in\Spec(A,\cL_n)$ with $z_n\to z\in\rho(A)$ (see Example \ref{ex1} and \cite{boff,bost,daug,rapa,lesh}).
It is also quite possible that for a normal operator we can encounter sequences $z_n\in\sigma(S_n)$ with $z_n\to z\in\rho(A)$. From
Corollary \ref{lim2b} it follows that this phenomenon can only occur if $z\in\mathcal{Q}(\sigma_{\ess}(A))$ or
$\overline{z}\in\sigma(A)\backslash\mathcal{Q}(\sigma_{\ess}(A))$ and $z\in\rho(A)$. For self-adjoint operators this does not represent a problem since \eqref{lower} ensures that
all erroneous limit points are non-real; however, normal operators can have non-real points in the spectrum.


\begin{corollary}
Let $(\cL_n)\in\Lambda$, $z\in \sigma(A)\backslash\mathcal{Q}(\sigma_{\ess}(A))$, $\delta(\cL(\{z,\overline{z}\}),\cL_n) =\varepsilon_n$ and fix $r,\Gamma$ as above. For a sequence $z_n\in\sigma(S_n)$, with $z_n\to z$, we set
\begin{displaymath}
\gamma_n(z_n) = \min\{\Vert(A-z_n)\phi\Vert:\phi\in\mathcal{M}_n(\{z\},r)^\pm,~\Vert\phi\Vert = 1\},
\end{displaymath}
then $\dist[z_n,\sigma(A)]\le \gamma_n(z_n)  = \mathcal{O}(\varepsilon_n)$, and for all sufficiently large $n\in\mathbb{N}$ we have
$\vert z_n- z\vert\le \gamma_n(z_n)$.
\end{corollary}
\begin{proof}
It suffices to show that $\gamma_n(z_n)  = \mathcal{O}(\varepsilon_n)$ and this is an immediate consequence of Theorem \ref{lim3} and
Corollary \ref{susp}.
\end{proof}

If we now define the following limit set
\begin{displaymath}
\hat{\lim_{n\to\infty}}\sigma(S_n) = \big\{z\in\mathbb{C}:\textrm{ there exist } z_n\in\sigma(S_n)\textrm{ with }z_n\to z \textrm{ and }\gamma(z_n)\to 0\big\},
\end{displaymath}
we obtain
\begin{displaymath}
\Big(\hat{\lim_{n\to\infty}}\sigma(S_n)\Big)\backslash\mathcal{Q}(\sigma_{\ess}(A)) = \sigma_{\dis}(A)\backslash\mathcal{Q}(\sigma_{\ess}(A)).
\end{displaymath}

\section{Unbounded Operators}
We suppose now that $A$ is an unbounded normal operator and that $\alpha\in\rho(A)\cap\mathbb{R}$. We define the following norm on $\Dom(A)$:
$\Vert\phi\Vert_{A} = \sqrt{\Vert A\phi\Vert^2 + \Vert\phi\Vert^2}$.
For a $\psi\in\Dom(A)$ and a subspace $\cL\subset\Dom(A)$
we write $\dist_A[\psi,\mathcal{L}] = \inf\{\Vert \psi - \phi\Vert_A:\phi\in\mathcal{L}\}$. If a sequence of subspaces $(\cL_n)$ satisfies $\dist_A[\psi,\cL_n]\to 0$ for all $\psi\in\Dom(A)$ then we write $(\cL_n)\in\Lambda(A)$. For two subspaces $\cL,\mathcal{M}\subset\Dom(A)$ let
\begin{displaymath}
\delta_A(\cL,\mathcal{M}) = \sup_{\psi\in\cL,~\Vert\psi\Vert_A=1}\dist_A[\psi,\mathcal{M}].
\end{displaymath}

The idea of mapping the second order spectrum of a bounded operator to that of an unbounded operator was introduce in
\cite[Lemma 3]{bost} and used to prove that for a self-adjoint operator $A$ with $(a,b)\cap\sigma(A)\subset\sigma_{\dis}(A)$
we have
\begin{equation}\label{dislim}
\Big(\lim_{n\to\infty}\sigma(S_n)\Big)\cap\mathbb{D}(a,b) =  \sigma_{\dis}(A)\cap(a,b)\quad\textrm{for all}\quad(\cL_n)\in\Lambda(A)
\end{equation}
(see \cite[Corollary 8]{bost}). We will use this mapping idea to extend our convergence results to unbounded normal operators.

For a basis $\{\psi_1,\dots,\psi_m\}$ of $\cL_n\subset\Dom(A)$ we have the matrices $B$, $L$, $M$,
and $S_n$ defined by \eqref{matrices0} and \eqref{matrices2}. Consider also the following matrices
\begin{displaymath}
B(\alpha)_{i,j} = \langle (A-\alpha)\psi_j,(A-\alpha)\psi_i\rangle,\quad
L(\alpha)_{i,j} = \langle(A+A^*-2\alpha)\psi_j,\psi_i\rangle,\quad
M(\alpha)_{i,j} = \langle \psi_j,\psi_i\rangle,
\end{displaymath}
and
\begin{displaymath}
S_n(\alpha) = \left(
\begin{array}{cc}
M(\alpha)^{-1} & 0\\
0 & M(\alpha)^{-1}
\end{array} \right)\left(
\begin{array}{cc}
L(\alpha) & -B(\alpha)\\
M(\alpha) & 0
\end{array} \right).
\end{displaymath}
Now we set $\hat{\psi}_j = (A - \alpha)\psi_j$, define the subspace $\hat{\cL}_n = \Span\{\hat\psi_1,\dots\hat\psi_m\}$, and note that
$\dim\hat{\cL}_n = \dim \cL_n$ follows from the fact that $\alpha\in\rho(A)$. Consider the matrices
\begin{displaymath}
\hat{B}(\alpha)_{i,j} = \langle (A-\alpha)^{-1}\hat\psi_j,(A-\alpha)^{-1}\hat\psi_i\rangle,
\quad \hat{L}(\alpha)_{i,j} = \langle((A-\alpha)^{-1}+(A^* - \alpha)^{-1})\hat\psi_j,\hat\psi_i\rangle,\quad
\hat{M}(\alpha)_{i,j} = \langle \hat\psi_j,\hat\psi_i\rangle,
\end{displaymath}
so that $\hat{B}(\alpha) = M(\alpha) = M$, $\hat{L}(\alpha) = L(\alpha) - 2\alpha M$ and $\hat{M}(\alpha) = B(\alpha) =
B - \alpha L + \alpha^2M$. Each of the matrices $\hat{B}(\alpha),\hat{L}(\alpha)$ and $\hat{M}(\alpha)$ defines an operator on $\hat\cL_n$ in a natural way:
\begin{align*}
\hat{B}(\alpha)\psi = \sum_{i}\langle (A-\alpha)^{-1}\psi,(A-\alpha)^{-1}&\hat{\psi}_i\rangle\hat{\psi}_i,\quad
\hat{L}(\alpha)\psi = \sum_{i}\langle[(A-\alpha)^{-1}+(A^*-\alpha)^{-1}]\psi,\hat{\psi}_i\rangle\hat{\psi}_i,\\
&\textrm{and}\quad \hat{M}(\alpha)\psi = \sum_{i}\langle \psi,\hat{\psi}_i\rangle\hat{\psi}_i.
\end{align*}
Now consider the block operator matrix
\begin{displaymath}
\hat{S}_{n}(\alpha) := \left(
\begin{array}{cc}
\hat{M}(\alpha)^{-1} & 0\\
0 & \hat{M}(\alpha)^{-1}
\end{array} \right)\left(
\begin{array}{cc}
\hat{L}(\alpha) & -\hat{B}(\alpha)\\
\hat{M}(\alpha) & 0
\end{array}\right):\hat\cL_n\oplus\hat\cL_n\to\hat\cL_n\oplus\hat\cL_n,
\end{displaymath}
and note that if $\hat{P}_n$ is the
orthogonal projection onto $\hat\cL_n$, then
\begin{displaymath}
\hat{S}_{n}(\alpha) = \left(
\begin{array}{cc}
\hat{P}_n[(A - \alpha)^{-1}+(A^*-\alpha)^{-1}] & -\hat{P}_n(A - \alpha)^{-1}(A^*-\alpha)^{-1}\\
I & 0
\end{array} \right):\hat\cL_n\oplus\hat\cL_n\to\hat\cL_n\oplus\hat\cL_n.
\end{displaymath}
Evidently, we have
\begin{align*}
\Spec_2((A - \alpha)^{-1},\hat{\cL}_n) &= \sigma(\hat{S}_n(\alpha))
= \{z^{-1}:z\in\sigma(S_n(\alpha))\}\\
&= \{z^{-1}:z\in\Spec_2((A-\alpha),\cL_n)\}= \{(z-\alpha)^{-1}:z\in\Spec_2(A,\cL_n)\}\\
&= \{(z-\alpha)^{-1}:z\in\sigma(S_n)\}.
\end{align*}

For a $z\in\sigma_{\dis}(A)$ with $(z-\alpha)^{-1}\notin\mathcal{Q}(\sigma_{\ess}((A - \alpha)^{-1}))$ and
\begin{displaymath}
\dist\bigg[(z-\alpha)^{-1},\Big(\mathcal{Q}(\sigma_{\ess}((A - \alpha)^{-1}))\cup\sigma((A - \alpha)^{-1})\cup
\sigma((A^*-\alpha)^{-1})\Big)\backslash\{(z - \alpha)^{-1}\}\bigg] = \delta,
\end{displaymath}
we denote by $\hat{\mathcal{M}}_n(\{(z - \alpha)^{-1}\},r)$ the spectral subspace of $\hat{S}_n(\alpha)$ associated
to those eigenvalues enclosed by a circle $\Gamma$ with center $(z - \alpha)^{-1}$
and radius $r>0$.  We will always assume that $r<\delta$, $\Gamma\cap\mathbb{R}=\varnothing$ if $z\notin\mathbb{R}$
and $\Gamma$ does not pass through zero if $z\in\mathbb{R}$.

For a $z\in\sigma_{\dis}(A)$, $\cL(\{z\})=\Span\{\phi_1,\dots,\phi_k\}$ and
$\cL(\{\overline{z}\})=\Span\{\phi_{k+1},\dots,\phi_{k+m}\}$ where the $\phi_j$ are orthonormal, we write
\begin{align*}
\mathcal{M}_\alpha(\{z\}) &= \Span\left\{\left(
\begin{array}{c}
(z-\alpha)^{-1}\phi_1\\
\phi_1
\end{array} \right),\dots,\left(
\begin{array}{c}
(z-\alpha)^{-1}\phi_{k+m}\\
\phi_{k+m}
\end{array} \right)\right\}\quad\textrm{if}\quad z\notin\mathbb{R}\\
\mathcal{M}_\alpha(\{z\}) &= \Span\left\{\left(
\begin{array}{c}
0\\
\phi_1
\end{array} \right),\left(
\begin{array}{c}
\phi_1\\
0
\end{array} \right),\dots,\left(
\begin{array}{c}
0\\
\phi_{k+m}
\end{array} \right),\left(
\begin{array}{c}
\phi_{k+m}\\
0
\end{array} \right)\right\}\quad\textrm{if}\quad z\in\mathbb{R}.
\end{align*}

\begin{theorem}\label{unbounded}
Let $(\cL_n)\in\Lambda(A)$, $z\in\sigma_{\dis}(A)$ with $(z-\alpha)^{-1}\notin(\mathcal{Q}(\sigma_{\ess}((A - \alpha)^{-1}))$,
$\delta_{A}[\cL(\{z,\overline{z}\}),\cL_n] = \varepsilon_n$ and fix $r,\Gamma$ as above.
Then $\hat\delta(\hat{\mathcal{M}}_n(\{(z-\alpha)^{-1}\},r),\mathcal{M}_\alpha(\{z\})) = \mathcal{O}(\varepsilon_n)$,
$\dist[z,\sigma(S_n)] = \mathcal{O}(\varepsilon_n)$ and $z$ is isolated in $\lim_{n\to\infty}\sigma(S_n)$.
\end{theorem}
\begin{proof}
First we show that $(\hat{\cL}_n)\in\Lambda$. Let $u\in\mathcal{H}$, then there exists a $\psi\in\Dom(A)$ such that
$(A-\alpha)\psi = u$. Since $(\cL_n)\in\Lambda(A)$ we have a sequence $\psi_n\in\cL_n$ with $u - (A-\alpha)\psi_n \to 0$,
and $(\hat{\cL}_n)\in\Lambda$ follows. Now let $\cL(\{z,\overline{z}\}) = \Span\{\phi_1,\dots,\phi_{k+m}\}$ where $\cL(\{z\}) = \Span\{\phi_1,\dots,\phi_k\}$, $\cL(\overline{z}\}) = \Span\{\phi_{k+1},\dots,\phi_{k+m}\}$, and the $\phi_j$ are
orthonormal. Since $\delta_{A}[\cL(\{z,\overline{z}\}),\cL_n] = \varepsilon_n$ there are vectors $\psi_{n,j}\in\cL_n$ with
$\Vert(A-\alpha)(\phi_j - \psi_{n,j})\Vert\le \varepsilon_n(1+\vert\alpha\vert)\sqrt{\vert z\vert^2 + 1}$
for each $1\le j\le k+m$. Set $\hat{\psi}_{n,j} = (A-\alpha)\psi_{n,j}\in\hat{\cL}_n$, then for any normalised
$\phi\in\cL(\{z,\overline{z}\})$ we have $\phi = \sum\langle\phi,\phi_j\rangle\phi_j$ and
\begin{align*}
\Big\Vert\phi - \sum_{j=1}^k\frac{\langle\phi,\phi_j\rangle}{z-\alpha}\hat{\psi}_{n,j} -
\sum_{i=k+1}^{k+m}\frac{\langle\phi,\phi_i\rangle}{\overline{z}-\alpha}\hat{\psi}_{n,i}\Big\Vert &\le
\Big\Vert\sum_{j=1}^k\langle\phi,\phi_j\rangle\Big(\phi_j - \frac{\hat{\psi}_{n,j}}{z-\alpha}\Big)\Big\Vert
+ \Big\Vert\sum_{i=k+1}^{k+m}\langle\phi,\phi_i\rangle\Big(\phi_i
 - \frac{\hat{\psi}_{n,i}}{\overline{z}-\alpha}\Big)\Big\Vert\\
&=
\Big\Vert(A-\alpha)\sum_{j=1}^k\frac{\langle\phi,\phi_j\rangle}{z-\alpha}(\phi_j - \psi_{n,j})\Big\Vert\\
&+ \Big\Vert(A-\alpha)\sum_{i=k+1}^{k+m}\frac{\langle\phi,\phi_i\rangle}{\overline{z}-\alpha}(\phi_j
 - \psi_{n,j})\Big\Vert\\
&\le \frac{(k+m)\varepsilon_n(1+\vert\alpha\vert)\sqrt{\vert z\vert^2 + 1}}{\vert z-\alpha\vert}.
\end{align*}
Therefore we have $\delta(\cL(\{z,\overline{z}\},\hat{\cL}_n)\le(k+m)\varepsilon_n(1+\vert\alpha\vert)\sqrt{\vert z\vert^2 + 1}/\vert z-\alpha\vert$. The assertions follow from an application of Theorem \ref{lim3}
to the operator $(A-\alpha)^{-1}$ and eigenvalue $(z - \alpha)^{-1}$.
\end{proof}

\begin{corollary}\label{last}
Let $A$ be a self-adjoint, $(\cL_n)\in\Lambda(A)$ and $z\in\sigma_{\dis}(A)$. There exists an $\alpha\in\rho(A)\cap\mathbb{R}$ such that
$(z-\alpha)^{-1}\notin\mathcal{Q}(\sigma_{\ess}((A - \alpha)^{-1}))$. Let $\dist_{A}[\cL(\{z\}),\cL_n] =
\varepsilon_n$, then $\hat\delta(\hat{\mathcal{M}}_n(\{(z-\alpha)^{-1}\},r),\mathcal{M}(\{z\})) = \mathcal{O}(\varepsilon_n)$ and
$\dist[z,\sigma(S_n)] = \mathcal{O}(\varepsilon_n)$.
\end{corollary}
\begin{proof}
If $z\in\sigma_{\dis}(A)$, then there exists a $\tau>0$ such that $(z-\tau,z+\tau)\cap\sigma(A) = \{z\}$ and we may choose any
$\alpha\in(z-\tau,z+\tau)\backslash\{z\}$. Since $z-\alpha\in\mathbb{R}$ and
$(z-\alpha)^{-1}\notin\sigma_{\ess}((A - \alpha)^{-1}))$,
it follows that $(z-\alpha)^{-1}\notin\mathcal{Q}(\sigma_{\ess}((A - \alpha)^{-1}))$.
\end{proof}

Combining Corollary \ref{last} with \eqref{dislim} we have the following statement: if $(a,b)\cap\sigma(A)\subset\sigma_{\dis}(A)$ and $(\cL_n)\in\Lambda(A)$, then we have
\begin{displaymath}
\Big(\lim_{n\to\infty}\sigma(Sn)\Big)\cap\mathbb{D}(a,b) =  \sigma_{\dis}(A)\cap(a,b),
\end{displaymath}
and for any $z\in(a,b)\cap\sigma_{\dis}(A)$ there exist $z_n\in\sigma(S_n)$
with $\vert z_n-z\vert = \mathcal{O}(\delta_{A}(\cL(\{z\}),\cL_n))$. Now let $(a',b')\cap\sigma(A) = \{z\}$, then using \eqref{high} we have
\begin{equation}\label{quadcon}
\left[\Re z_n - \frac{\vert\Im z_n\vert^2}{b' - \Re z_n},\Re z_n + \frac{\vert\Im z_n\vert^2}{\Re z_n - a'}\right]\cap\sigma(A) = \{z\}
\end{equation}
for all sufficiently large $n\in\mathbb{N}$. From \eqref{quadcon} it follows that $\vert \Re z_n - z\vert = \mathcal{O}(\delta_{A}(\cL(\{z\}),\cL_n)^2)$. The convergence rate in Corollary \ref{last} has been observed in computations (see \cite[examples 6 and 8]{bost}).

\section{Acknowledgements}
The author gratefully acknowledges the support of EPSRC grant no. EP/I00761X/1.

\end{document}